\documentclass[12pt]{amsart}
\usepackage{geometry}     
\usepackage{graphicx}
\usepackage{color}

\usepackage{amsthm}
\usepackage{amsmath}
\usepackage{amsfonts}
\usepackage{enumerate}
\usepackage{amssymb}
\usepackage{epstopdf}
\usepackage{hyperref}

\newtheorem{theorem}{Theorem}[section]
\newtheorem{proposition}[theorem]{Proposition}
\newtheorem{lemma}[theorem]{Lemma}
\newtheorem{corollary}[theorem]{Corollary}
\newtheorem{conjecture}[theorem]{Conjecture}
\newtheorem{conjectureo}[theorem]{Optimistic Conjecture}

\DeclareMathOperator{\Vol}{Vol}
\DeclareMathOperator{\M}{\mathbf{M}}
\newcommand{\R}{\mathbb{R}}

\newcommand{\Z}{\mathbb{Z}}
\newcommand{\N}{\mathbb{N}}
\newcommand{\Om}{\Omega}
\newcommand{\Si}{\Sigma}

\newcommand{\Zk}{\mathcal{Z}_k(M, \partial M; G)}
\newcommand{\ZkPL}{\mathcal{Z}_k^{PL}(M, \partial M; G)}

\newcommand{\Zone}{\mathcal{Z}_1}

\newcommand{\F}{\mathcal{F}}
\newcommand{\eps}{\varepsilon}

\theoremstyle{definition}

\newtheorem{remark}[theorem]{Remark}

\title{Parametric inequalities and Weyl law for the volume spectrum}
\author{Larry Guth and 
Yevgeny Liokumovich}

\begin{document}

\maketitle

\begin{abstract}
We show that the Weyl law for the volume spectrum in a compact 
Riemannian manifold 
conjectured by Gromov can be derived from parametric generalizations of two
famous inequalities: isoperimetric inequality and coarea inequality.
We prove two such generalizations in low dimensions and obtain the Weyl law for 1-cycles in 3-manifolds. We also give a new proof of the 
Almgren isomorphism theorem.
\end{abstract}

\section{Introduction}
Consider the space of $k$-dimensional Lipschitz cycles in a Riemannian manifold $M$
with coefficients in an abelian group $G$.
One can define the flat semi-norm $\mathcal{F}(x)$ on this space 
as the infimum of volumes of Lipschitz chains with boundary $x$;
define flat distance as $\mathcal{F}(x,y) = \mathcal{F}(x-y)$
(\cite{Flem}, \cite[Appendix 1]{Gu2}).
If we consider the quotient of this space
by the equivalence relation $x\sim y$ if $\mathcal{F}(x,y)=0$ and then take its completion,
then we obtain the space of flat k-cycles $\mathcal{Z}_k(M;G)$
with flat distance $\mathcal{F}$. This is a natural space to consider 
when trying to find submanifolds that solve a certain calculus of variations
problem (like minimal surfaces) without a priori specifying the topological type
of the solution. 

In the 1960's Almgren initiated a program of developing Morse theory
on the space $\mathcal{Z}_k(M;G)$. It follows from Almgren's work
that there is a non-trivial cohomology class $\alpha \in H^{n-k}(\mathcal{Z}_k(M;\Z_2);\Z_2)$ corresponding to a sweepout of $M$ by
an $(n-k)$-parameter family of $k$-cycles $F: X \rightarrow \mathcal{Z}_k(M;\Z_2)$
with $F^*(\alpha) \neq 0 \in H^{n-k}(X;\Z_2)$. Moreover, all cup powers $\alpha^p$
are also non-trivial and the corresponding families of cycles we call $p$-sweepouts.
Corresponding to each $p$ we can define the $p$-width
$$\omega_p^k(M) = \inf\{\sup_{x \in X} Vol_k(F(x)): F\text{ is a p-sweepout} \}$$

The $p$-widths correspond to volumes of certain generalized minimal submanifolds
that arise via a min-max argument. In dimension $k=n-1$, $3 \leq n \leq 7$, they are
smooth minimal hypersurfaces and for $k = 1$ they are stationary geodesic nets. 
The study of $p$-widths led
to many remarkable breakthroughs in recent years, including the resolution
of the Willmore Conjecture \cite{MNWil} and Yau's conjecture on existence
of infinitely many minimal surfaces \cite{IMN}, \cite{So}. (See also \cite{zh19}, \cite{ChMa}, \cite{MN19}
and references therein).

Gromov suggested to view these widths, or ``volume spectrum'', as
non-linear analogs of the eigenvalues of the Laplacian \cite{Gr88}, \cite{Gr03}, \cite{Gr09}. 
This framework is very useful for obtaining new results about widths and 
some other non-linear min-max geometric quantities (see, for example, \cite{GALio},
\cite{lio2016}, \cite{mazurowski}).
In line with this 
analogy Gromov conjectured that $\omega_p^k$'s satisfy an asymptotic Weyl law:
\begin{equation} \label{weyl_law_general}
    \lim_{p \rightarrow \infty} \frac{\omega_p^k(M)}{p^{\frac{n-k}{n}}}= a(n,k) \Vol(M)^{\frac{k}{n}}
\end{equation}
For $k=n-1$ this was proved in \cite{LMN}. 

In this paper we show that for $k<n-1$ Gromov's conjecture can be reduced to proving
parametric versions of two famous inequalities in geometry:
the isoperimetric inequality and coarea inequality. We prove these
parametric inequalities in low dimensions and as a consequence obtain
the Weyl law for 1-cycles in 3-manifolds.

\begin{theorem} \label{weyl3}
For every compact Riemannian 3-manifold $M$
$$\lim_{p \rightarrow \infty} \frac{\omega_p^1(M)}{p^{\frac{2}{3}}}= a(3,1) \Vol(M)^{\frac{1}{n}}$$
\end{theorem}

We do not know the value of constant $a(3,1)$.
In \cite{ChMa2} Chodosh and Mantoulidis proved
that optimal sweepouts of the round 2-sphere
are realized by zero sets of homogeneous polynomials and
that $a(2,1) = \sqrt{\pi}$, but for all other values of
$n>k \geq 1$ constants $a(n,k)$ remain unknown.

\subsection{Some applications of the Weyl law}
In \cite{IMN} Irie, Marques and Neves used the Weyl law to 
prove that for a generic Riemannian metric on an n-dimensional 
manifold, $3 \leq n \leq 7$, the union of embedded minimal hypersurfaces forms a 
dense set. Marques, Neves and Song proved a stronger equidistribution 
property of minimal hypersurfaces in \cite{MNS}; their proof is based on 
the idea that one can ``differentiate'' both sides of (\ref{weyl_law_general})
with respect to some cleverly chosen variations of metric. 
Similar results were proved by Gaspar and Guaraco using the Weyl law in the Allen-Cahn setting \cite{GaGu} (see also \cite{dey}). 
Li used the Weyl law
to prove existence of infinitely many 
minimal hypersurfaces for generic metrics in higher dimensions, $n>7$ \cite{li2019}.
In \cite{So} Song proved some Weyl law-type
asymptotic estimates for certain 
non-compact manifolds and used them to prove that in dimensions $3 \leq n \leq 7$ for all Riemannian metrics on a compact  manifold (not only generic ones) there exist infinitely many
minimal hypersurfaces (see also \cite{SoZh} where similar generalizations
were applied to prove a ``scarring'' result for minimal hypersurfaces).

In \cite{So19} Song showed that the density result for generic metrics 
can be obtained without the full strength
 of the Weyl law. This observation was used in \cite{LiokSt} (together with 
 a bumpy metrics theorem for stationary geodesic nets proved in \cite{St}) 
 to prove that
 for a generic Riemannian metric on a compact manifold $M^n$, $n \geq 2$, stationary geodesic nets form a dense set, even though the Weyl law 
 for 1-cycles in dimensions $n \geq 4$ is not known. In \cite{LiSt} Li and 
 Staffa used Theorem \ref{weyl3} to prove an equidistribution result for stationary geodesic nets in 3-manifolds, analogous to that of 
 \cite{MNS} for minimal hypersurfaces.

\subsection{Parametric isoperimetric inequality}
The isoperimetric inequality of Federer-Fleming
asserts that given a $k$-dimensional Lipschitz
cycle $z$ in $\R^n$, there exists a $(k+1)$-dimensional
Lipschitz chain $\tau$, such that $\partial \tau = z$ and
\begin{equation} \label{FF}
    Vol_{k+1}(\tau) \leq c(n) \Vol_{k}(z)^{\frac{k+1}{k}}
\end{equation}
This inequality is useful when $z$ lies in a ball of large radius.
If $z \subset B_R$ with $R \leq \Vol_{k}(z)^\frac{1}{k}$,
then the inequality simply follows by taking a cone over $z$.

In this paper we will be interested in the situation 
when a mod 2 cycle $z$ is contained in a small cube (say, of side length $1$)
and has very large volume. 
In this case, we can subdivide the cube into smaller cubes 
of side length $\Vol(z)^{-\frac{1}{n-k}}$. Applying the Federer-Fleming
deformation in this lattice we can find a ``deformation chain'' that pushes
$z$ into
the $k$-dimensional skeleton of this lattice. 
Since the total volume of the
$(k+1)$-skeleton of the lattice
is $\sim \big(\frac{1}{\Vol(z)^{\frac{1}{n-k}}}\big)^{k+1}\Vol(z)^{\frac{n}{n-k}}$ we can then obtain a filling $\tau$ of $z$ with
\begin{equation} \label{FFinsmall}
    \Vol_{k+1}(\tau) \leq c(n)  \Vol_{k}(z)^{\frac{n-k-1}{n-k}}
\end{equation}
Note that when $\Vol(z)>>1$ this is a much better bound
than (\ref{FF}) or the cone inequality.

We would like to prove a parametric version of inequality (\ref{FFinsmall}),
namely, given a family of relative cycles $F:X \rightarrow \mathcal{Z}_k([0,1]^n, \partial [0,1]^n;G)$ we would like to find a continuous family 
of $(k+1)$-chains $H:X \rightarrow I_{k+1}([0,1]^n;G)$
with $\partial H(x) - F(x)$ in $\partial [0,1]^n$ and so that the mass of $H(x)$
is controlled in terms of the mass of $F(x)$ as in (\ref{FFinsmall}).

First, observe that we need to assume that map $F$ is contractible, 
since otherwise a continuous family of fillings $H$
does not exist. Secondly, existence of a family of fillings
with controlled mass depends on the choice of coefficients $G$.
In Section \ref{sec:Almgren} we show that for every integer $N$ there exists
a contractible 1-parameter family of $0$-cycles in $\mathcal{Z}_0(S^1, \Z)$
of mass $\leq 2$, such that every continuous family of fillings
must contain a chain of mass $>N$. This example
can be generalized to higher dimensions and codimensions. 

However, for $G = \Z_2$ we can optimistically conjecture:

\begin{conjectureo}
Let $F:X \rightarrow \mathcal{Z}_k([0,1]^n, \partial [0,1]^n;\Z_2)$
be a contractible family. Then there exists a family of $(k+1)$-chains
$H:X \rightarrow I_{k+1}([0,1]^n;\Z_2)$, such that
$\partial H(x) - F(x)$ is supported in $\partial [0,1]^n$ and
$$\M(H(x)) \leq c(n)  \max\{1,\M(F(x))^{\frac{n-k-1}{n-k}}\}$$
\end{conjectureo}

The actual conjecture that we prove in the case of
low dimensions and use for the proof of the Weyl law is somewhat more technical.
We will be interested in the situation where maximal 
mass $\M(F(x)) \sim p^\frac{n-k}{n}$.
In this case we have
$$p^{\frac{n-k-1}{n}}+
   \M(F(x))p^{-\frac{1}{n}} \sim \M(F(x))^{\frac{n-k-1}{n-k}} $$
We also require that the family $F(x)$ is $\delta$-localized, a regularity condition that
can always be guaranteed after a small perturbation (see Section \ref{sec:approximation theorem}).
Finally, it will be convenient to state the conjecture 
for more general domains in $\R^n$ with piecewise smooth boundary satisfying 
certain regularity condition on the boundary (see Section \ref{theta-corners}). In particular, it is satisfied
if $\partial \Om$ is smooth. 

\begin{conjecture} \label{conj: isoperimetric}
Let $\Omega \subset \R^n$ be a connected domain, $\partial \Om$
 piecewise smooth boundary with $\theta$-corners, $\theta \in (0, \pi)$, $ 0 \leq k <n$. 
There exist constants $c(\Omega)>0$ and $\delta(\Omega,L, p)>0$ with the following property.
Let $F: X^p \rightarrow \mathcal{Z}_{k}(\Om, \partial \Om; \Z_2)$ be a continuous
$\delta(\Omega,L, p)$-localized contractible 
$p$-dimensional family with $\M(F(x)) \leq L$.
Then there exists map
 $H: X \rightarrow I_{k+1}(\Om;\Z_2)$, such that

\begin{itemize}
    \item $\partial H(x) - F(x) \subset \partial \Om$
    for all $x$;
    \item $\M(H(x)) \leq c(\Om) \big(p^{\frac{n-k-1}{n}}+
   \M(F(x))p^{-\frac{1}{n}} \big)$
\end{itemize}
\end{conjecture}

We also want to state a parametric isoperimetric inequality conjecture
for cycles in $\R^n$. It doesn't have immediate applications to the
Weyl law, but is interesting on its own. Unlike conjectures above we expect it
to be true for coefficients in every abelian group $G$.
For simplicity we state it for $G = \Z$.

\begin{conjecture}
Let $F: X^p \rightarrow \mathcal{Z}_k(\R^n; \Z)$ be a continuous 
family of cycles.
There exists a family of chains $H: X^p \rightarrow I_{k+1}(\R^n;\Z)$, such that $\partial H(x) = F(x)$ and
$$ \M(H(x))\leq c(n) \M(F(x))^{\frac{k+1}{k}}$$
\end{conjecture}

The main difficulty in this conjecture seems to be proving a bound for $\M(H(x))$ that is independent of $p$.
An inequality of this type with 
constant $c(n,p)$ that depends on $p$ subexponentially
would also be of interest.
In another direction, one can ask if it is possible to replace
$c(n)$ with $c(k)$ and extend this result to infinite-dimensional
Banach spaces as in \cite{gromov1983filling} (see also \cite{wenger})
or prove a similar inequality for the Hausdorff content 
as in \cite{LLNR}.

\subsection{Parametric coarea inequality}
One can similarly formulate a parametric version of the coarea inequality.
Let $M$ be a Riemannian manifold with boundary and
let $M_{\eps} = M \setminus cl(N_{\eps}(\partial M))$, $M$ minus closed $\eps$-neighborhood of its boundary. If $z$ is
a Lipschitz relative cycle in $M$ and $r_0>0$, then the
coarea inequality applied to the distance function
implies the existence of some $r \in [0,r_0]$, 
such that $\Vol(\partial (z \llcorner M_r)) \leq \frac{Vol(z)}{r_0}$.

Very optimistically, one could conjecture that given a family of
relative cycles $\{z_x\}_{x \in X}$ and $r_0>0$ there exists 
$r \in [0,r_0]$, such that 
$\Vol(\partial (z_x \llcorner M_r)) \leq \frac{\Vol(z_x)}{r_0}$ for all
$x \in X$.

\begin{figure} 
	\centering
	\includegraphics[scale=1.2]{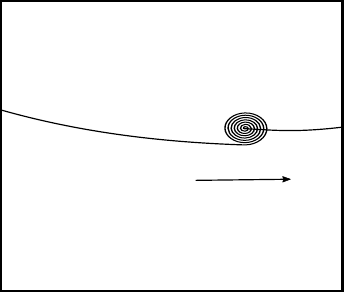}
	\caption{}
	\label{fig:example}
\end{figure}

This is not always the case. Consider a 1-parameter family 
$\{z_x\}_{x \in [0,1]}$
of relative cycles in $M = [0,1]^2$ with a small but 
tightly wound spiral moving from the center to the boundary of $[0,1]^2$ (Fig. \ref{fig:example}).
No matter how you choose $r$, for some value of $x \in [0,1]$ the spiral
will hit the boundary of $M_r$ and the intersection may 
have arbitrarily many points.
Letting $r$ vary as a continuous function of $x$ doesn't help.

However, what can help is if we allow ourselves to modify family $\{ z_x\}$
in the boundary $\partial M_{r}$, while only increasing
its mass by a small amount. It seems natural 
to conjecture that the amount of extra mass that we may have to add
should be less than or comparable to the size of an optimal $p$-sweepout
of the boundary by $k$-cycles, that is $\lesssim p^{\frac{n-k-1}{n-1}}$.
Note that this is much smaller than the mass of an optimal 
$p$-sweepout of the $n$-dimensional interior by $k$-cycles 
($\gtrsim p^{\frac{n-k}{n}}$).

Also, observe the following: if the family of relative cycles is a $p$-sweepout and
if it can be represented by a family of 
chains, such that their boundaries are a continuous family in $\partial M_r$,
then the family of boundaries is a $p$-sweepout of $\partial M_r$.
On the other hand, we know that an optimal $p$-sweepout of $[0,1]^2$ by $1$-cycles
has maximal mass of $const \, \sqrt{p}$, 
while every $p$-sweepout of 
$\partial [0,1]^2$ must have $0$-cycles of mass $\geq p$. 
It follows that if we want the family of boundary
cycles to be continuous, then $\frac{\M(z_x)}{r_0}$ bound is not sufficient.

Motivated by this we state the following optimistic conjecture:

\begin{conjectureo}
Let $n>k\geq 1$. Let $F:X^p \rightarrow \mathcal{Z}_k([0,1]^n, \partial [0,1]^n;\Z_2)$
be a continuous family of cycles and $r \in (0, \frac{1}{2})$.
Then there exists a family $F':X \rightarrow I_k([r_0,1-r_0]^n;\Z_2)$, 
such that
\begin{enumerate}
    \item $F'(x)\llcorner (r,1-r)^n = F(x)\llcorner (r,1-r)^n$;
    \item $\M(F'(x)) \leq \M(F(x)) + c(n)p^{\frac{n-k-1}{n-1}}$;
    \item $\partial F'(x)$ is a continuous family of 
    $(k-1)$-cycles in $\partial [r_0,1-r_0]^n$;
    \item $\M(\partial F'(x))\leq c(n) \big(\frac{\M(F(x))}{r} + p^{\frac{n-k}{n-1}} \big)$
\end{enumerate}
\end{conjectureo}

In our more down-to-earth conjecture that we use to prove
the Weyl law we allow the $F'(x)$ to be a small perturbation 
of $F(x)$ in the interior of $(r,1-r)^n$.
In addition,
we require that the family satisfies a certain technical assumption
that the mass does not concentrate at a point. We also state the result for more general domains
with piecewise smooth boundary with $\theta$-corners (their definition is
given in Section \ref{theta-corners}).

\begin{conjecture} \label{conj: coarea}
Let $\Omega \subset \R^n$ be a connected domain, $\partial \Om$
 piecewise smooth boundary with $\theta$-corners, $\theta \in (0, \pi)$.
Fix $\eta>0$.  For all $p\geq p_0(\Om)$, $r \in (0, r_0(\Om))$ and $\delta>0$ the following holds.
Let $F: X^p \rightarrow \mathcal{Z}_k(\Omega,\partial \Omega; \Z_2)$ be
a continuous map with no concentration of mass.
There exists a map $F': X \rightarrow 
I_k(cl(\Omega_r); \Z_2)$, such that 

\begin{enumerate}
\item $\partial F'(x)$ is a continuous $\delta$-localized family
in $\mathcal{Z}_{k-1}(\partial \Om_{r}; \Z_2)$;
    \item $\mathcal{F}(F(x)\llcorner  \Om_r, F'(x)) < \eta$;
\item $\M(F'(x)) \leq \M(F(x)) + \frac{\M(F(x))}{r}p^{-\frac{1}{n-1}} +C(\Om) \M(\partial \Om)^{\frac{k}{n-1}} p^{\frac{n-k-1}{n-1}}$;
\item  $\M(\partial F'(x)) \leq c(\Om) ( \frac{\M(F(x))}{r} + p^{\frac{n-k}{n-1}})$.
\end{enumerate}

Moreover, if $F$ is a $p$-sweepout of $\Omega$, then 
$\partial F': X \rightarrow \partial \Om_{r}$
is a $p$-sweepout of $\partial \Om_{r}$
by $(k-1)$-cycles.
\end{conjecture}

We note that the parametric coarea inequality Conjecture for $k$-cycles in $n$-domains together with 
parametric isoperimetric inequality for $(k-1)$-cycles
in $(n-1)$-domains can be used to prove 
parametric isoperimetric inequality for $k$-cycles in $n$-domain.

\subsection{Weyl law} In the proof of the Weyl law for
$(n-1)$-cycles in a compact $n$-manifold in \cite{LMN}
manifold $M$ is subdivided 
into small domains $U_i$ and then a family of hypersurfaces $\{ \Sigma_t \}$ is defined
that restricted to each $U_i$ is a $p$-sweepout $\{ \Sigma_t^i = \Sigma_t \cap U_i \}$ of $U_i$.
However, these hypersurfaces have boundary in $\partial U_i$
and do not form a family of cycles in $M$ (only a family of 
chains). A crucial step in the proof is to glue this family of chains
into a family of cycles in $M$ by adding a family of $(n-1)$-chains
inside $\bigcup \partial U_i$. Since we are using $\mathbb{Z}_2$ coefficients their total volume is bounded by the volume of $\bigcup \partial U_i$
and is negligible compared to $p^{\frac{1}{n}}$ for $p$ very large. 

In the case of higher codimension we can similarly define
a family of $k$-chains, which restrict to a $p$-sweepout
of $U_i$ by relative cycles. 
However, to carry out the second step we need to find a family 
of chains in $\bigcup \partial U_i$ of controlled volume that we can add to 
$\Sigma_t$ to turn it into a family of cycles. The problem is
twofold:
\begin{enumerate}
    \item the volumes of $\partial \Sigma_t^i$ can be arbitrarily large;
    \item even if we find a way to control the volume of $\partial \Sigma_t^i$,
    we also need to construct a continuous family of chains filling $\partial \Sigma_t^i$
    of volume that is negligible compared to the maximal volume of $\Sigma_t^i$.
\end{enumerate}
Together the parametric coarea inequality and parametric isoperimetric inequality 
solve these two problems.

\subsection{Organization} In Section \ref{sec:approximation}
we prove some important results on how families of cycles can be approximated
with families that have some nice special properties that we call
$\delta$-localized families. In Section \ref{sec:Almgren} we use these results
to give a new simpler proof of the Almgren isomorphism theorem.
In Section \ref{sec:coarea} we prove a parametric coarea inequality for
1-cycles in 3-dimensional domains. In Section \ref{sec:isoperimetric}
we prove a parametric isoperimetric inequality for 0-cycles in 2-dimensional
domains. In Section \ref{sec:weyl} we show how the Weyl law can be proved
using the parametric coarea and isoperimetric
inequality conjectures.

\begin{remark}
    After an earlier draft of this paper appeared on the arxiv Bruno 
    Staffa proved Conjectures \ref{conj: coarea} and \ref{conj: isoperimetric} for $k=1$ and all $n$ \cite{StaCoarea}, \cite{StaWeyl}. Consequently, Staffa obtained the Weyl law Theorem \ref{weyl general} for $1$-cycles in all dimensions.
\end{remark}

\subsection{Acknowledgements} The second author was partially supported by NSF grant, NSERC Discovery grant and Accelerator Award and Sloan Fellowship. The second author would like to thank Alexander Kupers, Fernando Marques and Andre Neves for stimulating discussions. We are grateful to the anonymous referee for numerous corrections to the first draft of this article.
We would like to thank Bruno Staffa for finding a gap in the proof of 
Proposition \ref{extension} in an earlier draft of this paper and suggesting a way to fix it.

\section{Approximation results 
for families of cycles} \label{sec:approximation}

\subsection{Piecewise smooth boundaries with corners} \label{theta-corners}
It will be convenient to work with manifolds that have 
piecewise smooth boundaries with no cusps. In this subsection we
give a more precise definition of the ``no cusps'' property.

Let $M$ be a Riemannian manifold with piecewise smooth boundary.
Given a subset $Y \subset M$ let $
N_{r}(Y) = \{x: dist(x, Y)<r \}$ denote the $r$-neighbourhood of $Y$.

Given $\theta \in (0, \frac{\pi}{2})$ we will say that
$\partial M$ is a boundary with $\theta$-corners if
for every $p \in \partial M$ and every $\eps>0$
there exists an open ball $B \ni p$
and a $(1+\eps)$-bilipschitz diffeomorphism $\Phi: B \rightarrow \R^n$,
such that $\Phi(p) = 0$ and $\Phi(\partial M \cap B)$
lies in a union of at most $n$ hyperplanes $P_1, ..., P_j$ in general position (that is, the intersection of any $k$ hyperplanes is a liner subspace of codimension $k+1$),
s.t. the dihedral angle $\angle(P_i, P_{i'}) \geq \theta$
for all $i \neq i'$.

We will need the following useful lemma.

\begin{lemma} \label{lem: boundary with corners}
For every $\eps>0$ there exist $r_0>0$ and $(1+\eps)$-Lipschitz
maps $E: \partial M \times [0,r_0] \rightarrow U$, $N_{\frac{9r_0}{10}}(\partial M) \subset U
\subset N_{\frac{11r_0}{10}}(\partial M)$,
such that $E$ is the identity map on $\partial M$.
\end{lemma}

\begin{proof}
If $\partial M$ is smooth, the we can choose $E$
to be the exponential map $E = exp_{\partial M}$. 

Otherwise, consider a stratification of $\partial M$ 
into $0-$, $1-$, ..., $(n-1)$-dimensional smooth strata $\{S_i\}$.
Inductively, 
we can define a continuous vectorfield $V$ on $\partial M$,
such that 
\begin{itemize}
    \item $V$ is smooth and $|V(x)|=1$ on $S_i$ for each $i$;
    \item $\angle(V(x), \partial M) \geq \frac{\theta}{3}$.
\end{itemize}
We can extend $V$ to a smooth vectorfield in the interior
of $M$ and define $E:\partial M \times [0,r] \rightarrow N_r(\partial M)$
to be a smooth map with $\frac{\partial E}{\partial t}|_{(x,t)=(x,0)} = V(x)$.
Then for $r$ sufficiently small this map will be the desired
bilipschitz diffeomorphism.
\end{proof}

Let $r_0>0$ and $E$ be as in the lemma above. For any $r \in (0, r_0)$ 
let $M_r = M \setminus E(\partial M \times [0,r))$.

Define a projection map $\Phi_1^r: M \rightarrow M_r $
given by $\Phi_1^r(E(x,t))= E(x,r)$ if $x \in M \setminus M_r$
and $\Phi_1^r(x) = x$ otherwise. 
Observe that $\Phi_1^r$ is $(1+\eps)$-Lipschitz.

We have the following consequence of the
coarea formula for flat chains.

\begin{theorem} \label{classical coarea}
Let $\tau \in I_{k}(M; G)$ then there exists a constant $c(\eps)$
with $\lim_{\eps \rightarrow 0} c(\eps) = 1$, so that for every $\eps>0$
there exists $r \in (0, \eps)$, such that
$\tau \llcorner M_r \in I_{k}(M_r; G)$ and
$$\M(\tau \llcorner \partial M_r)) \leq c(\eps) \frac{\M(\tau)}{\eps}$$
Moreover, if $\tau \in \mathcal{Z}_k(M, \partial M; G)$,
then $\tau \llcorner M_r \in \mathcal{Z}_k(M_r, \partial M_r; G)$
and $\partial (\tau \llcorner M_r) \in \mathcal{Z}_{k-1}(\partial M_r; G)$.
\end{theorem}

\subsection{Cubical complexes}
We will be considering families of flat chains
parametrized by a cubical complex. We start with some definitions
that will be convenient when working with these complexes,
following \cite{LMN}.

 Given $m\in\N$ let $I^m$ denote 
 the $m$-dimensional cube $[0,1]^m$. 
 For each $j\in \N$, $I(1,j)$ denotes the cube complex on $I^1$  whose $1$-cells and $0$-cells (vertices) are, respectively,  
$$[0,3^{-j}], [3^{-j},2 \cdot 3^{-j}],\ldots,[1-3^{-j}, 1]\quad\mbox{and}\quad [0], [3^{-j}],\ldots,[1-3^{-j}], [1].$$
We denote by $I(m,j)$  the cell complex on $I^m$: 
$$I(m,j)=I(1,j)\otimes\ldots \otimes I(1,j)\quad (\mbox{$m$ times}).$$
Then $\alpha=\alpha_1 \otimes \cdots\otimes \alpha_m$ 
is a $k$-cell of $I(m,j)$ if and only if $\alpha_i$ is a cell
of $I(1,j)$ for each $i$, and $\sum_{i=1}^m {\rm dim}(\alpha_i) =k$. 
We will abuse notation by identifying a $k$-cell 
$\alpha$ with its support: $\alpha_1 \times \cdots \times \alpha_m \subset I^m$. 

Let $X$ be a cubical subcomplex of $I(m,j)$.
For $i \geq j$ let $X(i)$ denote a subcomplex of $I(m,i)$,
whose cells lie in a cell of $X$. In other words, $X(i)$
is obtained from $X$ by subdividing each $k$-cell of $X$
into $3^{q(i-j)}$ smaller cells.
We use the notation $X(i)_k$ to denote the set of all $k$-cells in $X(i)$. 
If $E$ is a cell of $X(i)$, to simplify notation we will 
write $E_k$ to denote the $k$-skeleton of the cell $E$
(dropping $(i)$) and $E(i')_k$ to denote the $k$-skeleton of $X(i+i') \cap E$.
We write $X^p$ to denote a cubical subcomplex with
cells of maximal dimension $p$.


\subsection{Approximation theorem.} \label{sec:approximation theorem}
Let $G = \mathbb{Z}$ or $\mathbb{Z}_p$.
Assume that $M$ is a manifold with piecewise smooth boundary $\partial M$
with $\theta$-corners for some $\theta \in (0, \frac{\pi}{2})$.
Let $\Zk$ denote
the space of relative cycles with coefficients in group $G$.
Let $\mathcal{F}$ denote the flat distance in $\Zk$.

Each individual cycle in a continuous family $\{ F(x) \}$ of flat cycles 
can be well-approximated by a polyhedral cycle.
However, close to every nice looking cycle in the family there could
be a sequence of some wilder and wilder looking wiggly cycles converging to it. For example their supports may converge in Hausdorff topology to the whole space.
This makes it hard to apply some continuous deformations
or surgeries to the family.
It is desirable to have an approximation theorem
which tells us that we can pick a sufficiently fine discrete subset 
of the parameter space, and the values of $F$ on that discrete subset 
uniquely determine the homotopy class of the family, and, moreover, if
we apply some continuous deformation to this discrete family so that it 
satisfies a certain property,
we could interpolate to obtain a continuous family, which also 
satisfies this property. 

This gives rise to discrete and continuous settings in Min-Max Theory
(see e.g. \cite{Pitts}, \cite{MNWil}, \cite{MNRicci}).

The main result of this section is the following approximation theorem.

\begin{theorem} \label{approximation}

(1) There exists $c=c(M,p)>0$, such that
the following holds. For every $\eta>0$ there exists $\varepsilon_0>0$, such that if $F: X^p(q)_0 \rightarrow \Zk$ is a map
with $\mathcal{F}(F(a),F(b)) < \eps_0$ for every pair of adjacent vertices 
$a$ and $b$, then there exists an extension of $F$ to a continuous map 
$F: X \rightarrow \Zk$, such that 
for each cell $C$ in $X(q)$
and $x,y \in C$ we have
$\mathcal{F}(F(x),F(y))< c \eta \sup_{w \in X(q)_0} \M(F(w))  $. 


(2) There exists $\eps(M,p)>0$, such that if $F_0: X \rightarrow \Zk$ 
and $F_1: X \rightarrow \Zk$ 
are two maps
with $\mathcal{F}(F_0(x),F_1(x)) < \eps$ for all $x$ and $\sup_{x\in X} \M(F(x))< \infty$, $i=0,1$, then $F_0$ and $F_1$ are homotopic.
\end{theorem}

Informally, Theorem \ref{approximation} means that 
we can replace a continuous family of cycles 
with a sufficiently dense (in the parameter space)
discrete family. We can then perturb each of the cycles 
in the discrete family and complete them to a new continuous
family. This new continuous family will be homotopic
to the original family.
In general, our procedure may increase the mass of cycles by 
a constant $c(p,M)$.
In the last subsection we show that if we assume that the family of cycles satisfies
no concentration of mass condition (and without this condition in the case of $0$-cycle), then
the mass will only
increase by an amount that goes to $0$ as $\eps$ goes to $0$.

Our method is to first modify the family so that for
a small ball $B$ in the parameter space and any two points
$x, y \in B$, the difference $F(x) - F(y)$ lies in some fixed
(and depending only on $B$) collection of disjoint
convex sets. We will call such families ``$\delta$-localized". Then we can interpolate between $F(x)$
and $F(y)$ using radial contraction inside these convex sets.

\subsection{Admissible collections of open sets in \texorpdfstring{$M$}{TEXT}} 

Pick $r_0 < injrad (M)$.
If $M$ has smooth boundary let
$E_{\partial M}: \partial M \times [0,r_0] \rightarrow M$
denote the normal exponential map from the boundary of $M$.
More generally, let $E_{\partial M}$ denote map $E$
from Lemma \ref{lem: boundary with corners}.

We say that a collection of open set $U_l \subset M$
is $r$-admissible if they are all disjoint
and 

1) if $U_l$ is disjoint from $\partial M$,
then $U_l$ is a ball of radius $r_l$;

2) if $U_l$ is not disjoint from $\partial M$,
then $U_l= E_{\partial M}(B_{r_l} \times [0,r_l])$;

3) $\sum r_l < r$.

We need the following two elementary lemmas about collections of $r$-admissible sets.

\begin{lemma} \label{admissible}
Let $\mathcal{B}_1 = \{\beta_i \}$ be a finite set of balls
in $\partial M$ and $\mathcal{B}_2 = \{ B_j \}$ be a finite set of disjoint
balls contained in the interior of $M$.
There exists an $r$-admissible collection
$\{U_l \}$, such that
$\cup U_l \supset (\cup E_{\partial M}(\beta_i \times [0,rad(\beta_i))) \cup (\cup B_j)$
and $r \leq \sum rad(\beta_i) + 3 \sum rad(B_j)$.
\end{lemma}

\begin{proof}
If $\beta_{i_1}$ and $\beta_{i_2}$ intersect,
then we can find a ball $\beta \subset \partial M$ of radius $\leq rad(\beta_{i_1}) + rad(\beta_{i_2})$
that contains their union. 
Hence, we may replace $\{\beta_i \}$ with a collection of disjoint balls
containing the union of original balls without increasing the
sum of their radii. Let's call this replacement operation (*).

If $B_j $ intersects $E_{\partial M}(\beta_i \times [0,rad(\beta_i)])$, then 
there exists $\beta'$ of radius $\leq rad(\beta_i)+ 3 rad(B_j)$,
so that $B_j \subset E(\beta' \times [0,rad(\beta')])$.
(Note that instead of factor $3$ we could've used $2+\eps$
for $\eps \rightarrow 0$ as $rad(B_j) \rightarrow 0$.)
We define new sets $\mathcal{B}_1' = \mathcal{B}_1 \setminus \{ \beta_i \} \cup \{ \beta' \}$
and $\mathcal{B}_2' =  \mathcal{B}_2 \setminus  \{ B_j \}$.
By (*) we can replace $\mathcal{B}_1'$ with a collection of disjoint ball
$\mathcal{B}_1''$.
Let (**) denote the operation of replacing $\mathcal{B}_1$ and $\mathcal{B}_2$
with $\mathcal{B}_1''$ and $\mathcal{B}_2'$.
Note that (**) decreases the number of balls in $\mathcal{B}_2$
by $1$, decreases the sum of radii of balls in $\mathcal{B}_2$ by $rad(B_j)$ and increases 
the sum of radii of balls in $\mathcal{B}_1$ by $3 rad(B_j)$.

We perform (**) repeatedly, until $E_{\partial M}(\beta_i \times [0,rad(\beta_i)])$ and  $B_j$ are all disjoint.
The process terminates since the number of balls in $\mathcal{B}_2$
was finite.
\end{proof}

\begin{lemma} \label{two admissible}
Let $\{U^1_l\}_{l=1} ^{v_1}$ be an $r_1$-admissible
and $\{U^2_l\}_{l=1}^{v_2}$ be an $r_2$-admissible 
collection of open sets.
 Then there exists an $r$-admissible
collection $\{ U_l \}_{l=1}^{v}$ with $v\leq v_1+v_2$, $\sqcup U^1_l \cup \sqcup U^2_l \subset \cup U_l$
and $r \leq 3  (r_1 + r_2 ) $.
\end{lemma}

\begin{proof}
We replace the collection of all 
$\{ U^1_{l_i} \}$ and $U^2_{l_j}$, which are balls in the interior of
$M$, with a collection $\{ B_i \}$ of disjoint balls that contain their union. 
We can do this with $\sum rad(B_i) \leq \sum rad(U^1_{l_i}) + \sum rad(U^2_{l_j})$.

Some of balls $B_i$ may intersect $\partial M$. We replace them with 
$exp(\beta_i \times [0, r']) \supset B_i$ for $\beta_i \subset \partial M$
and $r' \leq 3 rad(B_i)$.

Now we are in the situation where we can apply Lemma \ref{admissible}.
\end{proof}

\subsection{Localized families}

Let $V \subset X(q)$.
We will say that a map $F: V \rightarrow \Zk$
is $\eps$-fine if for every cell $C$ in $X(q)$ and $a,b \in C \cap V$ we have
$\mathcal{F}(F(a),F(b)) \leq \eps$.
We will say that a map $F: V \rightarrow \Zk$
is $\delta$-localized if for every cell $C$ of $X(q)$
there exists a $\delta$-admissible collection of open sets $\{U_l^C \}$,
such that for every pair $a, b \in C \cap V$
we have the support $supp(F(a) - F(b)) \subset \bigcup U_l^C$.
Given a cell $C$ in $X(q)$ let $C_0= C \cap X(q)_0$.

\begin{proposition} \label{discrete to continuous}
There exists $\delta_0(M,p)>0$, $c(n, k,p) >0$ and $\eps_0(M,p)>0$ with the following property.
Let $F: X^p(q)_0 \rightarrow \Zk$
be a $\delta$-localized map, $\delta < \delta_0$.

I. There exists an extension
map $F:  X(q) \rightarrow \Zk$,
such that

(1) the extension is $c \delta$-localized;

(2) for every cell $C$ and $x,y \in C$
we have $$\mathcal{F}(F(x),F(y))< c \delta \max_{w \in C_0} \M(F(w) \llcorner \cup U^C_l)$$

(3) for every cell $C$ and $x \in C$ we have
 $$\M(F(x)) \leq \max_{w \in C_0} \M(F(w)) + c \max_{w \in C_0} \M( F(w) \llcorner \cup U^C_l)$$

II. Suppose in addition that $F$ is defined on some subcomplex
$Y(q) \subset X(q)$ and $F: Y \rightarrow \Zk$ is $\delta$-localized
and $\eps$-fine, $\delta< \delta_0, \eps< \eps_0$.
Then there exists an extension
map $F:  X(q) \rightarrow \Zk$,
such that

(1) the extension is $c \delta$-localized;

(2) for every cell $C$ in $X(q)$ and $x,y \in C$
we have $$\mathcal{F}(F(x),F(y))< c \delta \sup_{w \in C_0 \cup (C \cap Y)} \M(F(w))$$
and $\M(F(x)) \leq c \max_{w \in C_0 \cup (C \cap Y)} \M(F(w))$.

\end{proposition}

\begin{proof}
We prove I. Part II follows with some straightforward modifications.

We start by extending to 1-skeleton.



Let $C$ be an $m-$dimensional cell.
Assume we have extended $F$ for all $x \in \partial C$.
For each $(m-1)$-dimensional cell $C' \subset \partial C$
assume, inductively, that there exists
a $c(m-1) \delta$-admissible collection
$\mathcal{B}(C') = \{U_1^{C'}, ... , U_{l(C')}^{C'} \}$
with $supp(F(x) - F(y)) \subset \cup U_i^{C'} $ 
for all $x, y \in C'$.

By Lemma \ref{two admissible} there exists $c(m) >0$ and
a
$c(m) \delta$-admissible collection of open sets 
$\mathcal{B}(C) = \{U_1^{C}, ... , U_{l(C)}^{C} \}$,
such that 
$$ \cup_{i, C' \subset \partial C} U_i^{C'}
\subset \cup_i U_i^C$$

For each $i$ we define contraction maps
$R^i_t$ as follows.
If $U_i^C=B_{r_i}(p)$ is a ball centered at $p$ we let $R^i_t: B_{r_i}(p_i) \rightarrow B_{t r_i}(p_i)$ be the radial contraction map.
If $U_i^C=E(\beta_{r_i}(p_i) \times [0,r_i] )$ 
we define $R_t^i(E(x, r)) = E(tx,tr)$


Let $v(C)$ denote the vertex in $0$-skeleton
of $C$ that is closest to $0$ (thinking of
$C$ as the subset of the unit cube).

Define 
function $\sigma(x,t)$ on $\partial C \times [0,1]$ by 
$$\sigma (x,t) = F(x) + \sum_i R_{t}^i ((F(v(C)) - F(x)) \llcorner U_i^C)$$
Since $\sigma(x,1) = F(v(C))$ for all $x \in \partial C$
this gives a well-defined extension of $F$
to $C$.

Now we would like to verify the bounds on $\F$ and 
the mass.
By taking a cone over $F(x) - F(y) \subset \cup U^C_l$
we verify
$\F(F(x), F(y)) \leq \delta (\M(F(x) \llcorner \cup U^C_l)+\M(F(y)\llcorner \cup U^C_l))
\leq c(p) \delta \max_{w \in C_0} \M(F(w)  \llcorner \cup U^C_l) $.
Inductively we can check that 
$$\M(F(x)) \leq \max_{w \in C_0} \M(F(w)) + c(p) \max_{w \in C_0} \M (F(w)  \llcorner \cup U^C_l)$$
for sufficiently large $c(p)$.
\end{proof}

The following proposition is the key technical result necessary for the proof of
Theorem \ref{approximation}. It asserts that every sufficiently fine (in the flat topology)
discrete family of cycles can be approximated by a $\delta$-localized
discrete family of cycles. 

\begin{proposition} \label{extension}
Given $\delta >0$, there exists $\eps_0(p,M,\delta)>0$
with the following property.
Let $f: X^p(q)_0 \rightarrow \Zk$
be a map 
with
$\mathcal{F}(f(x),f(y))< \eps$
for any two adjacent vertices $x$ and $y$ and $\eps \in (0, \eps_0)$.

Then there exists $\tilde{q}(p,M, \delta)$ and a map
$F:  X(q+\tilde{q})_0 \rightarrow \Zk$
that coincides with $f$ on $X(q)_0$,
such that for every cell $C$ in $X(q)$ we have
\begin{enumerate}
    \item $\mathcal{F}(f(x),F(y))<
c(M,p, \delta) \eps$ for all $x \in C \cap X(q)_0$
and $y \in  C \cap X(q+\tilde{q})_0$;
\item
$\M(F(y)) \leq 
\sup_{w\in C \cap X(q)_0} \M(f(w)) + c(M,p,\delta)\eps$
for $y \in  C \cap X(q+\tilde{q})_0$;
\item $F$ is $\delta$-localized, moreover, if $E \subset C$ is a cell of $X(q+\tilde{q})$, then for the 
corresponding $\delta$-admissible collection $\{U^E_i\}_{i=1}^{n(p)}$ we have   $\M(F(y) \llcorner \bigcup U^E_i) \leq \sup_{w \in C \cap X(q)_0} \M(f(w)\llcorner \bigcup U^E_i)$.
\end{enumerate}
\end{proposition}

\begin{proof}
First we define $F$ on $X(q)_0$ setting
$F(x) = f(x)$.

We will define a sequence of numbers $q=q_0 < ... < q_p= \tilde{q}$
and maps 
$F_j : X(q+q_j)_0 \cap X(q)_j \rightarrow \Zk$ for $j = 0, ..., p$.
$F_p: X(q+\tilde{q})_0 \rightarrow \Zk$ will be the desired localized 
extension.

We organize our proof us follows. First, we collect all definitions
that will be needed in the proof. 
Then we prove a key lemma constructing a 
localized family of chains filling a given localized family of cycles.
Finally,
we give a proof of the inductive step of the construction.

\vspace{0.2in}

\textbf{Definitions and notation}

\begin{itemize}
    \item Let $r_1(M)$ be sufficiently small, so that
    every ball of radius $r_1$ that lies in $M$
is $1.01$-bilipschitz diffeomorphic to an $r_1$-ball in $\R^{n-1} \times \R_{\geq0}$,
    and for every ball $\beta(r_1)$ in $\partial M$ we have that
    $E_{\partial M}$ is a $1.01$-bilipschitz diffeomorphism on $\beta(r_1) \times [0,r]$.
\item Let $r_0(M, \delta) = \min\{\frac{1}{2}r_1,  \delta\}$ and consider
 a collection of sets
$\{B_1, ..., B_N \} $ covering $M$, such that
each $B_i$ is either a closed ball of radius $r_0$,
or $B_i = E_{\partial M}(\beta_i, [0,r_0])$,
where $\beta_i$ is a closed ball of radius $r_0$ in $\partial M$.
We will call such sets ``generalized balls''.
Note that we can choose covering $\{B_i\}$, so that the number $N$ of generalized balls in the covering 
satisfies $N \leq c(M) \frac{1}{ \delta^n}$.
Let $t B_i$, $t>0$, denote the concentric generalized ball with the same 
center point and the radius $tr_0$.

    
    
\end{itemize}    

Given a chain $\tau$ by the coarea inequality Theorem \ref{classical coarea} there exist 
generalized balls $B_i'$, $ B_i \subset B_i' \subset 2 B_i$, such that 
\begin{equation} \label{eq:coarea}
    \M(\tau \llcorner \partial B_i') \leq \frac{2}{r_0} \M(\tau \llcorner 2 B_i)
\end{equation}
It will be convenient to use $\{ B_i' \}$ to define a collection of sets with disjoint interiors as follows:
we let $U_1 = B_1'$ and let $U_i$ be the closure of $ B_i' \setminus \bigcup_{j=1}^{i-1} B_j'$ for $i>1$.
We will say that $\{U_i\}$ is a coarea covering for $\tau$.
Note that $U_i \subset B_i' \subset 2 B_i$ for each $i$, and 
$$\M(\tau \llcorner \bigcup_{i=1}^N \partial U_i) \leq \M(\tau \llcorner \bigcup_{i=1}^N \partial B_i') \leq \frac{2N}{r_0} \M(\tau) $$

More generally, given a finite collection $\{\tau_j\}_{j=1} ^{L}$, applying coarea inequality to $\sum \tau_j$, there exists
a collection of generalized balls 
$B_i'$, $ B_i \subset B_i' \subset 2B_i$, such that for each $j$ we have
\begin{equation} \label{eq:coarea2}
    \M(\tau_j \llcorner \partial B_i') \leq \frac{2L}{r_0} \max_{i=1,..,L} \M(\tau_i \llcorner 2 B_i)
\end{equation}
As above we set $U_1 = B_1'$ and $U_i = B_i' \setminus \bigcup_{j=1}^{i-1} B_j'$ for $i>1$
and  call $\{ U_i\}$ a coarea covering for $\{\tau_j \}_{i=1} ^{L}$. We then have 
\begin{equation} \label{eq:coarea3}
    \M(\tau \llcorner \bigcup_{i=1}^N \partial U_i) \leq \frac{2NL}{r_0} \max_{j=1,..,L} \M(\tau_j )
\end{equation}

\begin{itemize}
    
    \item 
    It will be convenient to define a way of chopping away 
a chain $\tau_j$ by intersecting it with sets
from the coarea covering $\{U_i\}$ of $\tau_j$.
Given a collection of chains $\{ \tau_j \}_{j=1}^L$ and a coarea covering $\{U_i \}_{i=1}^N$ for $\{\tau_j \}$
we let  
\begin{equation} \label{def of d_i}
  d_i(\tau_j, \{U_i \}) = \tau_j - \tau_j \llcorner \bigcup_{j=1}^{i}U_j  
\end{equation}

We then have the following mass bound for the boundary of $\partial d_i(\tau_j)$:
\begin{align} \label{mass d_i}
    \M(\partial d_i(\tau_j)) 
& \leq \M(\partial \tau_j) - \sum_{l=1}^i \M(\partial \tau \llcorner U_l) + \frac{2NL}{r_0} \M(\tau_j)
\end{align}
For $i > N$ define $d_i(\tau)$ to be the empty chain.

\item
For each cell $C \subset X(q)$
let $x(C) \in C \cap  X(q)_0$ denote the 
vertex of $C$ that is closest to $(0,..,0)$ 
(thinking of $X(q)$ as contained in the
ambient unit cube of large dimension).

    \item $dist_{\infty}(x,y) = \min_j\{|x_j - y_j|\}$

\item
For a face $E^l \subset X(q)_l$ we define
$Center(E)$ to be the collection of all points $x$ in $E$
with $dist_{\infty}(x, \partial E) > 1/3^{q+1}$.



\item It will be useful to define maps between 0-skeleta of subdivisions of cells with nice properties. These constructions are similar to \cite{MNWil}[Appendix C]. 

Given a cell $E$ of $X(q)$ and $x \in E$ we let $[x]_{q'}$
denote the closest point of $E(q')_0$ to $x$; if there is more than one vertex of $E(q')_0$ that minimizes $d_{\infty}(x, \cdot)$,
then we set $[x]_{q'}$ to be the vertex closest to $(0,..,0)$ in the ambient cube $I^m \supset X \supset E$.


\item We will say that a map $G:Y(q')_0 \rightarrow Y(q'')_0$, $q'\geq q''$, is an adjacency-preserving extension if 
\begin{enumerate}
\item $G(x) = x$ for all $x \in Y(q'')_0$;
    \item for any two $x,y \in Y_1(q')_0$ that lie in the same face of
$Y_1(q')$ we have that $G(x), G(y)$ are in the same face of $Y_2(q'')$.
\end{enumerate}

\item Given a cell $E$ of $Y(q)$ and $q'>q''$
define $H_{q',q''}: E(q')_0 \rightarrow E_0(q'')_0$ by
$H_{q',q''}(x) = [x]_{q''}$. We can extend this definition to a map $H_{q',q''}: X(q+q')_0 \rightarrow X(q+q'')_0$
by inductively applying $H_{q',q''}$ on each face.
We observe that $H_{q',q''}$ is an adjacency-preserving extension. It will be convenient to write $H_{q'}$
for $H_{q',0}$.

\item Let $L_t(E) = \{ x \in E:d_\infty(x, \partial E)=t\}$ and
let $\phi_t: L_t(E) \rightarrow \partial E$ be a bijective homothety from the
center point of $E$. We define maps
$$P: E(q'+1)_0 \setminus Center(E) \rightarrow \partial E(q')_0 \times I^1(q')_0$$
$$Q: \partial E(q')_0 \times I^1(q')_0 \rightarrow E(q')_0 \setminus Center(E) $$
by setting 
\begin{align}
    P_{q'}(x) &= ([L_{d_\infty(x, \partial E)}(x)]_{q'},3^{q+1}d_\infty(x, \partial E))  \label{defP} \\
    Q_{q'}(x,t) & = [L_{t/3^{q+1}}^{-1}(x)]_{q'} \label{defQ}
\end{align}
Note that the composition map $$Q \circ P(x): E(q'+1)_0 \setminus Center(E) \rightarrow E(q')_0 \setminus Center(E)$$
is an adjacency-preserving extension. Moreover, 
we have $Q \circ P(x) = H_{q'+1,q'}(x)$ for $
x \in \partial E(q')_0$.


\end{itemize}

\begin{lemma} \label{small filling}
    There exists $\eps_0(k,j,M, \delta)>0$ with the following property. 
    Let $F:Y^j(q)_0 \rightarrow \Zk$
    be a $\delta$-localized 
    map with $
    \F(F(x),0)< \eps$ for all $x \in Y^j(q)_0$ and some $\eps< \eps_0(k,j,M, \delta)$.
    Then there exists $\overline{q}(k,j,M, \delta)$,
    $c_1(k,j,M)$, $c_2(k,j,M, \delta)$ 
    and map $\tau: Y^j(q+\overline{q})_0 \rightarrow I_{k+1}(M;G)$, such that $\tau$ is $c_1\delta$-localized, $\M(\tau(x))< c_2 \varepsilon$ 
    and $\partial \tau(x) = F \circ H_{\overline{q}}(x)$.
\end{lemma}

\begin{proof}
    The proof is by induction on $n-k$.
    If $n=k$, then taking $\tau(x)=0$ gives the desired result.

    Assume the lemma holds for families of $(k+1)$-cycles.
    To prove the result for families of $k$-cycles we now proceed
    by induction on $j$. 
    
    If $j=0$ the result is immediate.
    Assume the lemma holds for $(j-1)$-dimensional families.
    Applying it to the $(j-1)$-skeleton $Y(q)_{j-1}$ of $Y(q)$ we obtain 
    $c_1(k,j-1,M)\delta$-localized family 
    $$\tau_{j-1}: Y(q)_j \cap Y(q+q_{j-1})_0 \rightarrow I_{k+1}(M;G)$$
with $\partial \tau (x) = F\circ H_{q_{j-1}}(x)$ and a mass bound
    $$\M(\tau_{j-1}(x)) \leq c_2(k,j-1,M, \delta) \varepsilon$$

    Let $C^j$ be a $j$-dimensional face of $Y(q)$.
    Let $$Q: \partial C^j(q_{j-1})_0 \times I^1(q_{j-1})_0 \rightarrow C(q_{j-1})_0 \setminus Center(C)$$
    be the map defined as in (\ref{defQ}).
    We will now construct a map 
$$\tau_{\partial C \times I}:\partial C^j(q_{j-1})_0 \times I^1(q_{j-1})_0 \rightarrow I_{k+1}(M;G)$$
   We set $\tau_{\partial C \times I}(x,0) =  \tau_{j-1}\circ Q(x)$.
   
    Since $F$ is $\delta$-localized
    there exists a $\delta$-admissible collection of sets $\{U^C_i\}$
    with $  F(x)-F(x(C) )$ supported in $\bigcup U^C_i$ for all
    $x \in C_0$. 
   Hence, there exists a $(k+1)$-chain $fill(x)$ supported in $\bigcup U^C_i$
    with $\M (fill(x))< 2\varepsilon$ and $\partial fill(x) = F(x)-F(x(C)) $ for all $x \in C_0$ .
    We let
    $$\tau_{\partial C \times I}(x,1) = \tau_{j-1}(x(C))+fill(H_{q_{j-1}}(x))$$
    We define $$\tilde{\tau}(x) = \tau_{\partial C \times I}(x,1)-\tau_{\partial C \times I}(x,0)$$
From the definition we have that $\{ \tilde{\tau}(x) \} _{x \in \partial C(q_{j-1})_0}$ is a family of
relative cycles.
    
    Since  $\{ \tau_{\partial C \times I}(x)\}_{x \in \partial C(q_{j-1})_0}$ is $c(k,j-1,M)\delta$-localized, applying Lemma \ref{two admissible}
    we have that $\{ \tilde{\tau}(x) \}$ is
    $c'\delta$-localized for $c'< 20c_1(j-1,k,M)$.
We also have the mass bounds
$$\M(\tilde{\tau}(x))\leq  \M(\tau_{\partial C \times I}(x,0))+\M(\tau_{\partial C \times I}(x,1) ) \leq (2 c_2(j-1,k,M,\delta) +1) \varepsilon$$
Hence, we can apply the inductive assumption for families
of $(k+1)$-cycles. We obtain that 
    there exists a $c_1(k+1,j-1,M)c'\delta$-localized family $\{\sigma(x)\}_{x \in \partial C(q_{j-1}+q' )_0}$ of $(k+2)$-chains,
    such that $\partial \sigma(x) = \tilde{\tau} \circ H_{q_{j-1}+q',q_{j-1}}(x)$.
Let $\{U_i\}$ be a coarea covering for
    the family $\{ \sigma(x)\}_{x \in \partial C(q_{j-1}+q' )_0}$.
    
 Increasing $q'$ if necessary we may assume
    that $q_{j-1} + q' \geq \log_3(N)+1$.
   Let $n(t) = 3^{q_{j-1}+q'+1}dist_\infty(t, \partial C \times \{1\})$.
    For $t \in I(q_{j-1}+q')_0 $ we have
    $$\tau'(x,t) =\tau_{\partial C \times I}(x,0) + \partial d_{n(t)}(\sigma(x), \{U_i\})$$
Observe the following properties of $\tau'$:
\begin{enumerate}
    \item $\partial \tau'(x,t) = F_{j-1}\circ H_{q_{j-1} + q'}(x)$ for $t <1$;
    \item $\tau'(x,0) = \tau_{\partial C \times I}(x)$;
    \item  $\tau'(x,1) = \tau_{\partial C \times I}(H_{q_{j-1}+q',q_{j-1}}(x),1)
    = \tau_{j-1}(x(C))+fill(H_{q_{j-1}+q'}(x))$.
\end{enumerate}

By Lemma \ref{two admissible} we also have that this family is $c_1(k,j,M)\delta$-localized for $c_1(k,j,m)$
bounded in terms of $c_1(k,j-1,M)$, $c_1(k+1,j-1,M)$
and our choice of $r_0$.
From (\ref{mass d_i}) we have that 
$$\M(\partial \tau'(x,t)) \leq c_2(k,j-1,M, \delta) \varepsilon+
c_2(k+1,j-1,M,\delta) \frac{2N3^{q'}}{r_0}\varepsilon$$
Since $r_0$ was determined by $\delta$ and $M$ and our choice of $q'$ depended only on $k$, $j$ and
$M$, we can bound this expression by $c_2(k,j,M,\delta)\varepsilon$.
    
Hence, we've obtained a $c_1\delta$-localized family $$\{\tau'(x,t)\}_{{(x,t) \in \partial C(q_{j-1}+q')_0\times I(q_{j-1}+q')_0 }}$$
We let $\overline{q}=q_{j-1}+q'+1$.
    We can then define a family $\tau$ on $C(\overline{q})_0 \setminus Center(C)$ by setting $\tau(x) = \tau' \circ P$
    for the map $P:C( \overline{q})_0 \rightarrow \partial C(q_{j-1}+q')_0\times I(q_{j-1}+q')_0$.
    For $x \in Center(C) \cap C(\overline{q})_0$ 
    we set $\tau(x) = \tau_{j-1}(x(C)) + fill(H_{\overline{q}}(x))$.
    It follows from the above estimates that $\{ \tau(x) \}$ is a $c_1(k,j,M) \delta$-localized family
    with $\M(\tau(x)) \leq c_2(k,j,M,\delta) \varepsilon$.
    Defining $\tau$ in this way on each $j$-face we obtain the desired family.
\end{proof}

We can now prove Proposition \ref{extension}. In calculations below let $c_1$ 
denote a constant that only depends
on $M$, $p$ and $k$. The value of $c_1$ may change from line to line.

We will define a sequence of  $c_1\delta$-localized  maps $F_j: X(q)_j \cap X(q_j)_0 \rightarrow \Zk$
with $q_0 = q$ and satisfying
$F_{j}(x)=F_{j-1} \circ H_{q_j, q_{j-1}}(x)$
for $x \in X(q)_j \cap X(q+q_{j-1})_0$ and such that for every cell $E^l$ of $X(q)$, $l \leq j-1$, we have
\begin{enumerate}[i]
  \item 
$\mathcal{F}(f(x),F(y))<
c(M,j-1, \delta) \eps$ for $x \in E \cap X(q)_0$ and $y \in E \cap X(q+q_{j-1})_0$; \label{flat distance bound}
\item
$\M(F(y)) \leq 
\sup_{w\in E \cap X(q)_0} \M(f(w)) + c(M,p,\delta)\eps$
for $y \in  E \cap X(q+q_{j-1})_0$; \label{global mass bound}
\item if $D \subset E$ is a cell of $X(q+q_{j-1})$ and $\{U^D_i\}$ is the corresponding $c_1\delta$-admissible collection, then $\M(F(y) \llcorner \bigcup U^D_i) \leq \sup_{w \in E \cap X(q)_0} \M(f(w)\llcorner \bigcup U^D_i)$ for $y \in  E \cap X(q+q_{j-1})_0$. \label{local mass bound}
\end{enumerate}

First we set $F_0(x) = F(x)$ for $x \in X(q)_0$.
Assume we have defined $F_{j-1}$ on $X(q)_{j-1} \cap X(q_{j-1})_0$.
For each $j$-dimensional face $C$ of $X$, by 
Lemma \ref{small filling}, there exists a $c_1\delta$-localized family $\{ \tau(x) \}_{x \in \partial C(q_{j-1}+\overline{q})_0}$,
such that $\partial \tau(x) = F(x(C))-F_{j-1}\circ H_{q_{j-1}+\overline{q}, q_{j-1}} (x)$.
We let $q_j = q_{j-1}+\overline{q}+1$ and define 
\begin{equation} \label{def F_j}
    F_j(x) = F_{j-1} \circ H_{q_j, q_{j-1}}(x) = F \circ H_{q_j}(x)
\end{equation}on 
$X(q)_{j-1} \cap X(q_{j})_0$. We now fix $j$-dimensional cell $C$ and describe the extension
of $F_j$ from $\partial C \cap X(q_j)_0$ to $C \cap X(q_j)_0$.

Let $\{U_i\}$ be a coarea covering for the family $\{ \tau(x)\}_{x \in \partial C(q_{j-1}+\overline{q})_0}$.
To obtain the desired mass bound for $F_j$ we will need the following construction. We defined $F_j$ by composing $F_{j-1}$ with an adjacency preserving extension on $\partial C$ and now we'd like to extend to the interior of $C$
by ``chopping away" $\tau(x)$ using (\ref{def of d_i}), taking the boundary
of the resulting $(k+1)$-chain and adding it to the value of $F_j$ on the boundary, until we deform the
family of $\{F_j(x)\}_{x \in \partial C(q_j)_0}$
to the constant family equal to $F(x(C))$ on $Center(C)$. The problem with this
approach is that the restriction $F_j(x)\llcorner U_i$ 
 for $x \in \partial C(q_j)_0$ can be large for some values of $i$, while $F(x(C))\llcorner U_i$ can be large for other values of $i$, so the interpolation may increase the total mass
to more than $\M(F(x(C))+\M(F(y))$ (for some $y \in C_0$), violating mass bounds (\ref{global mass bound}) and (\ref{local mass bound}) above. Instead, we will make a mass-minimizing choice in each $U_i$ that will guarantee that the mass bounds are
satisfied.

For each $i=1,...,N$ we let $m_i = \tau(x(i)) \llcorner U_i$,
where $x(i) \in \partial C(q_{j-1}+\overline{q})_0$ is such that $$\M( (\partial \tau(x(i)) \llcorner U_i -F(x(C)) \llcorner U_i) \leq \M((\partial \tau(x) \llcorner U_i -F(x(C)) \llcorner U_i)$$ for all $x \in \partial C(q_{j-1}+\overline{q})_0$.
Define
 $$b(\{ \tau(x)\}, F(x(C)),\{ U_i \}) = \sum_{i=1}^N m_i$$
 From estimate (\ref{mass d_i}) we have the following mass bounds for $b(\{ \tau(x)\}, F(x(C)),\{ U_i \})$ and its boundary:

\begin{align}
    \M(b(\{ \tau(x)\}, F(x(C)),\{ U_i \})) & \leq \sum_{x \in C(q_{j-1}+\overline{q})_0} \M(\tau(x)) \label{eq: mass of b 1} \\  
\begin{split} \label{eq: mass of b 2}
        \M(\partial b(\{ \tau_j \},F(x(C)),\{ U_i \}) - F(x(C)) & \leq \min_{x \in C(q_{j-1}+\overline{q})_0} \M(\partial \tau(x) - F(x(C))\\
    &+ \frac{2^{j}N3^{q_{j-1}+\overline{q}}}{r_0} \max_{x \in C(q_{j-1}+\overline{q})_0} \M(\tau(x)) 
\end{split}
\end{align}

We define a map $$\tau': \partial C( q_{j-1}+ \overline{q})_0 \times I(q_{j-1}+ \overline{q})_0 \rightarrow I_{k+1}(M;G)$$
as follows.
Let $n(t) = 3^{q_{j-1}+ \overline{q}}(1-t)$ (without any loss of generality we may assume
$N< 3^{q_{j-1}+ \overline{q}}$). We set
\begin{equation} \label{eq:sigma}
    \tau'(x,t) = \tau \circ H_{q_j, q_j-1}(x) - d_{n(t)}(\tau \circ H_{q_j, q_j-1}(x)- b(\{ \tau(x)\}, F(x(C)),\{ U_i \})) 
\end{equation}
Finally, we let 
$$F_j(x) = F(x(C))- \partial \tau' \circ P_{q_j-1}(x)$$
for $x \in C(q_j)_0 \setminus Center(C)$.
Observe that this definition coincides with (\ref{def F_j}) on $\partial C(q_j)_0$.
For $x \in C(q_j)_0 \cap Center(C)$ we define
$$F_j(x) = F(x(C))- \partial  b(\{ \tau(x)\}, F(x(C)),\{ U_i \})) $$
The flat distance bound (\ref{flat distance bound}) follows from
(\ref{eq: mass of b 1}); the mass bounds (\ref{global mass bound}) and (\ref{local mass bound}) follow from
(\ref{eq: mass of b 2}) and our construction of $b$. 
\end{proof}

\begin{remark}
Notice that if $M$ is a PL manifold,
then we can perform the above construction starting with a polyhedral approximation
$F(x)$ for each cycle $f(x)$, $x \in X(q)_0$. The construction in the proof
of Proposition \ref{extension} then gives a discrete family $F:X(q')_0 \rightarrow \ZkPL $ 
of polyhedral cycles and Proposition \ref{discrete to continuous} a continuous
family $F:X \rightarrow \ZkPL $ of polyhedral cycles. Hence, whenever it is convenient, 
we can replace a Riemannian metric with a PL approximation and a family of flat cycles with
an approximating family of polyhedral cycles. All relative $k$-dimensional cycles in our 
construction may be chosen so that $F(x) \cap \partial M = \partial F(x)$ 
and is a $(k-1)$-dimensional polyhedral cycle in $\partial M$.
\end{remark}

Combining Proposition \ref{discrete to continuous}
and \ref{extension} we obtain the following Proposition.

\begin{proposition} \label{dc2}
For all $\delta>0$ sufficiently small there exists $\eps_1>0$,
such that the following holds.
Given a map $F: X(q)_0 \rightarrow \Zk$,
which is $\eps_1$-fine, there exists $q'$
and a continuous extension $F: X(q+q') \rightarrow \Zk$, such that:

1. $F: X(q+q') \rightarrow \Zk$ is $\delta$-localized;

2. for every cell $C$ of $X(q)$, $x \in C_0$ and $y \in C \cap X(q+q')_0$, we have  $\mathcal{F}(F(x),F(y))< c(M,p)   \delta \max_{w \in X(q)_0} \M(F(w) )$;

3. $\M(F(x)) \leq c(M,p)(1+\delta)\max_{w \in C_0} \M(F(w))$.
\end{proposition}

\begin{proof}
Follows by applying Propositions \ref{extension} and \ref{discrete to continuous}.
\end{proof}

Theorem \ref{approximation} (1) directly follows from Proposition \ref{dc2}.
To prove Theorem \ref{approximation} (2) we need the following lemma.

\begin{lemma} \label{discrete homotopy}
For every sufficiently small $\delta>0$ there exists $\varepsilon_2(M,p,\delta)>0$
such that the following holds.
Let $F_0: X(q) \rightarrow \Zk$ and $F_1: X(q) \rightarrow \Zk$ be two 
$\delta$-localized maps.
Suppose $\F(F_0(x), F_1(x)) \leq \eps \leq \eps_2$.
Suppose $$m = \max \{\sup_{w \in X} \M(F_0(w)), \sup_{w \in X} \M(F_1(w)) \}< \infty$$
Then there exists a homotopy $F: X(q) \times I(1, q) \rightarrow \Zk$
between $F_0$ and $F_1$, such that

1. $\F(F(x,t),F_0(x)) < c(p,M) m\delta+ \eps_2$ for all $x \in X$;

2. $\M(F(x,t)) \leq c(p,M) m $ for all $(x,t) \in X \times I$.
\end{lemma} 

\begin{proof}
First we will define $F$
restricted to $X(\tilde{q}) \times \{0\} \cup  X(\tilde{q})_0 \times I(1,\tilde{q})_0 \cup X(\tilde{q}) \times \{1\}$
for some $\tilde{q} \geq q$ using the construction from Proposition
\ref{extension}.
Then we will apply Proposition \ref{discrete to continuous} II
to extend $F$ to a continuous map defined everywhere.

Let $F|_{X(q) \times {0}} = F_0$
and $F|_{X(q) \times {1}} = F_1$.
Applying the inductive step extension argument from the proof of Proposition 
\ref{extension} we define $c\delta$-localized map $F^1$ on
$v \times I(q_1)_0$ for each $v \in X(q)_0$,
such that $F^1(v , t)=F(v,t)=F_t(v)$ for $t =0,1$.
Note that $\F(F(v,t), F_0(v))\leq \eps_2$.
Using the same extension argument we can then extend $F^1$ to a $\delta$-localized map
on $(X(q+q_1)_0 \times I(q_1)_0) \cap (X(q)_1 \times I(q_1)_0 )\setminus X(q) \times \{0,1\} $. If we define $F^1(x) = F(x)$ on $X(q) \times \{0,1\}$,
then by Lemma \ref{two admissible} $F^1$ will still be $c\delta$-localized.
This defines $F^1$ on $(X(q)\times I(q))_1 \cap 
(X(q_1)\times I(q_1))_0$. We apply the inductive step 
from the proof of Proposition \ref{extension} to define $F^i$, $i=2,...,p$, until we obtain a $c\delta$-localized
map on $X(q_p)_0 \times I(q_p)$. 
We can then apply Proposition \ref{discrete to continuous} II
to extend $F$ to a continuous map defined everywhere.
The mass and flat norm estimates follow from the corresponding estimates in Proposition \ref{discrete to continuous} and the inductive step in the proof
of Proposition \ref{extension}. 
\end{proof}

Let $F_0$ and $F_1$ be two maps as in Theorem \ref{approximation} (2). 
Let $m = \sup_{w \in X} \M(F_0(w)) + \sup_{w \in X} \M(F_1(w)) $.

By Proposition \ref{dc2} for every sufficiently small $\delta>0$ 
we can define $\delta$-localized maps $\overline{F}_0$ and $\overline{F}_1$,
which agree with $F_0$ and $F_1$ on the $0$-skeleton
of some subdivision $X(q')_0$.
By Proposition \ref{dc2} we then have
$$\mathcal{F}(\overline{F}_0(x), \overline{F}_1(x)) \leq c(M,p)m\delta + \varepsilon$$
For $\delta>0$ and $\varepsilon>0$ sufficiently
small we have that $\overline{F}_0$ and $\overline{F}_1$ are
homotopic by Lemma \ref{discrete homotopy}.

Now we would like to construct a homotopy between $\overline{F}_i$
and $F_i$ for $i=0,1$.

Let $\delta_i  $ be a decreasing sequence of positive numbers converging to $0$ and $\eta_i= \min\{ \delta_i, \eps_2(\delta_i)\}$, where $\eps_2(\delta_i)>0$ is from Proposition \ref{dc2}.
Let $q_i \geq q$ be an increasing sequence of integers,
so that $F_0: X(q_i) \rightarrow \Zk$ is $\eps_i$-fine.
Applying Proposition \ref{dc2} we obtain a sequence
of $\delta_i$-localized maps $G_i: X(q_i ') \rightarrow \Zk$.
We claim that $G_i \rightarrow F_0$. 
Given $x \in X$ let $x_i$ denote the closest point of
$X(q_i)_0$ to $x$.

\begin{align} \label{cauchy}
\begin{split}
    \F(G_i(x), F_0(x)) & \leq  \F(G_i(x), G_i(x_i)) + \F(G_i(x_i), F_0(x_i)) + \F(F_0(x_i), F_0(x)) \\
    & \leq c(M,p) m \delta_i+ \eps_i
\end{split}
\end{align}

The convergence follows by standard compactness results for flat cycles \cite[4.2.17]{federer}.

To finish the construction we need to define a homotopy
between $G_i$ and $G_{i+1}$.
Observe that if $G_i$ is $\delta_i$-localized with  
domain $X(q_i')$, then it is also $\delta_i$-localized
with domain $X(q_{i+1}')$ for $q_{i+1}' > q_i'$.
It follows from inequality (\ref{cauchy}) that 
$\F(G_{i+1}(x), G_i(x)) \rightarrow 0$. Assuming that $\delta_i$
is small enough we can apply 
Lemma \ref{discrete homotopy} to define homotopy $H_i(x,t)$
between maps $G_i:X(q_{i+1}') \rightarrow \Zk$ and 
$G_{i+1}: X(q_{i+1}') \rightarrow \Zk$ for all $i>0$. 
Moreover, the flat norm estimates from Proposition \ref{dc2} and Lemma \ref{discrete homotopy}
imply that for each $x$ and $t$ the sequence $H_i(x,t)$ is Cauchy in $\F$.
This implies continuity of homotopy from $\overline{F}_0$ to $F_0$.
The same construction works for the homotopy from $\overline{F}_1$ to $F_1$.

\subsection{Homotopies with better estimates for the mass.}

We need to prove an analogue of Theorem \ref{approximation}
with better estimate for the mass of cycles when the family satisfies some additional assumptions.

In case of $0$-cycles we prove approximation theorem with an optimal bound.

\begin{proposition} \label{approximation_0}
There exists $\eps_0(M,p)>0$ and $c(M,p)>0$, such that
the following holds.

Let $F: X(q)_0 \rightarrow \mathcal{Z}_0(M, \partial M; \Z_2)$ be a map
with $\mathcal{F}(F(a),F(b)) < \eps \leq \eps_0$ for every pair of adjacent vertices 
$a$ and $b$, then there exists an extension of $F$ to a continuous map 
$F: X \rightarrow \mathcal{Z}_0(M, \partial M; \Z_2)$ with the following properties.
For each cell $C$ in $X(q)$
and $x,y \in C$ we have

(a) $\mathcal{F}(F(x),F(y))< c \sup_{w \in X_0} \M(F(w)) \eps $ 

(b) $\M(F(x)) \leq \sup_{w \in X_0} \M(F(w))$.
\end{proposition}

\begin{proof}
We observe that every $\eps$-fine family of $0$-cycle 
$F: X(q)_0 \rightarrow \mathcal{Z}_0(M, \partial M; \Z_2) $ is
$\delta$-localized for some $\delta \leq c(p) \eps$.
Hence, there is no need for Proposition \ref{extension}.

We can prove an analogue of Proposition \ref{discrete to continuous}
with a better bound for the mass.
We do it as follows. 

First we extend to 1-skeleton as follows.
Let $E \subset X(q)_1$ and let $\{U^E_j \}$ be a
$\delta$-admissible collection of open sets
containing $F(x) - F(y)$, $\partial E = x-y$.
Therefore, we must have $\M(F(x) \llcorner U_j^E) = \M(F(y) \llcorner U_j^E) mod 2 $.
Let $z_j = 0$ if $\M(F(x) \llcorner U_j^E) = 0 mod 2 $
and  $z_j  = \{ c_l \} $  if $\M(F(x) \llcorner U_j^E) = 1 mod 2$,
where $c_j$ is a center point of $U_j^E$.

Let $e$ denote the midpoint of $E$. We define 
$$F(e) = F(x) - \sum F(x) \llcorner U_j^E + \sum z_j = F(y) - \sum F(y) \llcorner U_j^E + \sum z_j$$
We homotop $F(x)$ to the midpoint of $E$ by radially
 contracting $F(x) \llcorner U_j^E$ in each $U_j^E$. We do the same for $F(y)$.

Similarly, given a cell $C$ and a $\delta$-admissible collection
$\{U^C_j \}$, so that $F(x) - F(y) \subset \cup U^C_j$ for any $x,y \in \partial C$,
we define cycles $z_j$ and radially contract family $\{F(x) \}_{x \in \partial C}$
to $F(e) = F(x) -  \sum F(x) \llcorner U_j^C + \sum z_j$, where $e$ is the center point of $C$.
\end{proof}

For higher dimensional cycles we obtain a better bound
for families that have no concentration of mass.

Map $F: X \rightarrow \Zk$ is said to have no concentration of mass
if $$\limsup_{r \rightarrow 0} \{ \M(F(x) \llcorner B_r(p)): p \in M, x \in X \} = 0$$
for every $x \in X$ (see \cite[3.7]{MNRicci}).


\begin{proposition} \label{approximation no concentration}

Let $F: X(q) \rightarrow \Zk$ be a map with no concentration of mass and $\sup_{x \in X} \M(F(x))< \infty$.
For every $\eps>0$ and $\delta>0$ there exists a
$\delta$-localized map $F': X^p(q+\overline{q}) \rightarrow \Zk$, such that


(a)
$\mathcal{F}(F(x),F'(x))< \eps $;

(b)
 $\M(F'(x)) \leq \sup_{w \in X_0} \M(F(x))  + \eps $.

\end{proposition}

\begin{proof}
Let $m = \sup_{x \in X} \M(F(x))$.
Let $c_0>0$ be larger than the constant $c(n,p)$ appearing
in Proposition \ref{discrete to continuous}.
We will choose $\delta'>0$ and $\eps'>0$
satisfying the following conditions.

Conditions for $\delta'$:
\begin{enumerate}
\item $\delta'<\frac{\delta}{c_0}$;
    \item If $\{ U_l \}_{l=1}^{n(p)}$ is a $ \delta'$-admissible
collection of open sets constructed in Proposition \ref{extension}, then (using no concentration of mass condition)
we have $$\sup_{x \in X} \sum_{l=1}^{n(p)}\M(F(x) \llcorner \cup U_l)< \frac{\eps}{10 c_0}$$
\item $\delta'< \frac{\eps}{c_0m}$;
\item $\delta'< \frac{\delta}{c_0}$.
\end{enumerate}

Conditions for $\eps'$:
\begin{enumerate}
    \item  $\eps' < \eps_0(M,p,\delta')$,
    where $\eps_0$ is from the statement of
    Proposition \ref{extension};
    \item $\eps'< \frac{\eps}{10 c(M,p, \delta')}$,
     where $c(M,p, \delta')$ is the constant from
    Proposition \ref{extension}.
\end{enumerate}

Pick $q'>0$ large enough
so that $F:X(q+ q') \rightarrow \Zk$ 
is $\eps'$-fine. Apply Proposition \ref{extension} to $F$ restricted to 
$X(q+ q')_0$ to obtain (by our choice of $\eps'< \eps_0(M,p,\delta')$) a $\delta'$-localized
map $F': X(q+ q'+ \tilde{q})_0 \rightarrow \Zk$. 

Then we apply Proposition \ref{discrete to continuous}
to obtain a $c_0\delta'$-localized family.
By our choice of $\delta'$ we have that the family is
$\delta$-localized. It follows from Propositions \ref{extension} and
\ref{discrete to continuous} and our choices of $\eps'$, $\delta'$ that
for every $x \in C \subset X(q+ q'+\tilde{q})$, $w \in C \cap X(q+q')_0$,
$y \in C \cap X(q+q'+\tilde{q})_0$ we have

\begin{align*}
\mathcal{F}(F'(x),F(x))  & \leq \mathcal{F}(F'(x),F'(y)) +\F(F'(y),F'(w)) +
\mathcal{F}(F'(w),F(w)) +
\mathcal{F}(F(x),F(w))\\
& < c_0m\delta'+ c(M,p,\delta') \eps'+ \eps'< \eps \\
    \M(F'(x))&  \leq  \sup_x \M(F(x)) + c(M,p,\delta')\eps'+c_0 \max_{w \in C_0} \M(F(w) \llcorner \bigcup U_l^C)\\
& \leq \sup_x \M(F(x)) + \eps 
\end{align*}

\end{proof}

\section{Homotopy classes of the space of cycles} \label{sec:Almgren}
Here we use results from the previous section to give
a new proof of the Almgren isomorphism theorem \cite{A1}, \cite{A2}.

Let $X^p$ be a cubical subcomplex that is also 
a $p$-dimensional manifold and $G = \Z_2$ or $\Z$.
If $G= \Z$ we demand that $X$ is oriented.
Let $M$ be a Riemannian manifold and $\partial M$ be piecewise smooth
with $\theta$-corners. Fix a PL structure on $M$ that is a bilipschitz approximation of $M$.
To a $\delta$-localized map $F: X(q)_0 \rightarrow \ZkPL$
we associate 
a polyhedral $(p+k)$-cycle $A_F \in \mathcal{Z}_{p+k}^{PL}(X \times M; G)$ constructed as follows.

For each $l$-cell
$C^l \subset X(q)_l$ we will define a polyhedral $(l+k)$-chain 
$A(C) \in I^{PL}_{l+k}(X \times M; G)$. We will then define $A_F = \sum_{C } A(C)$,
where the sum is over all $p$-dimensional cells of $X(q)$.
The definition is by induction on $l$.
For each $x \in X$ let $i_x: M \rightarrow X \times M$
be the inclusion $i_x(a) = (x,a)$.
For $l = 0$ we define $A(x)=i_x(F(x))$ for each $x \in X(q)_0$.
For every $(l-1)$-cell $C^{l-1} \subset X(q)_{l-1}$ 
assume that we have already defined
$A (C)$. Given a set $U \subset M$ and a bilipshitz diffeomorphism $f: U \rightarrow \R^n$
with $f(U)$ convex, a point $x \in U$
and a relative cycle $B \in \mathcal{Z}_i(C \times U, C \times (U \cap \partial M); G)$ let
$$cone_x(B) = g^{-1}(Cone_{g(p_C,x)}(g(B)))$$
where $p_C$ is the center point of
cell $C$, $g=(id_C,f): C \times U \rightarrow C \times \R^n $ and $Cone_{g(p_C,x)}(g(B))$ denotes the cone over $g(B)$
in $C \times f(U)$ with vertex $(p_C,f(x))$.
We will choose $f$ to be the exponential map
from point $p$ if $p$ lies in the interior of $M$, or map $E$
from Section \ref{theta-corners} if $p \in \partial M$.

%


Define 
$A(\partial C) = \sum_{E \subset \partial C} (-1)^{i} A(E) $
where the sign $(-1)^i$ is defined in the standard way for the boundary operator
on a cubical complex, $\partial C =\sum_{E \subset \partial C} (-1)^{i} E$.
Let $\{U_i^C \}$ be an admissible collection of sets
corresponding to $\delta$-localized family $\{ F(x)\}_{x \in  C_0}$.
We fix a point $x_i \in U_i^C$ for each $i$; if $U_i^C$ intersects $\partial M$,
then we pick $x_i \in \partial U_i^C$.
We fix a vertex $v \in C_0$  and set 
$$A(C) = A(v) \times C + \sum_i cone_{x_i}([A(\partial C) - A(v) \times \partial C] \cap C \times U_i^C)$$
Observe that by induction and our definition of $\delta$-localized families
$A(\partial C) - A(v) \times \partial C$
is an $(l-1)k$-dimensional polyhedral relative cycle in $ \bigsqcup_i C \times U_i^C$,
so $A(C)$ is well-defined. 

It follows from the construction that $A$ has the following properties:
\begin{enumerate}
    \item  for an open and dense subset of points 
$x \in C$ and for an open and dense subset of points
$x \in \partial C$ 
we have that 
$$proj_M(proj^{-1}_X(x) \cap A(C))$$ is a polyhedral $k$-cycle;
\item there exists a sequence $q=q_0 < q_1 < q_2<...$
and $\eps_i \rightarrow 0$, s.t. 
discrete families of cycles $F^i: C(q_i)_0 \rightarrow \Zk$
given be $F^i(x) = proj_M(proj^{-1}_X(x) \cap A(C))$
are $c(l-1)\delta$-localized (with respect to a fixed  admissible collection
$\{U_i^C \}$ that depends only on $C$) and $\eps_i$-fine.
\end{enumerate}

Observe that if follows from the definition that if
$F_1$ and $F_2$ are two $\delta$-localized families
with $\mathcal{F}(F_1,F_2) < \eps$ for all $x \in X(q)_0$,
then for all sufficiently small $\eps>0$ 
the corresponding $(p+k)$-cycles $A_{F_1}$ and $A_{F_2}$
are homologous (cf. \cite[Section 1]{Gu1}). In particular, 
given a map $F: X \rightarrow \Zk$ we can define
$\alpha_{F} \in H_{p+k}(X \times M, X \times \partial M; G)$
corresponding to cycle $A_{F'}$, where $F'$ is
a $\delta$-localized approximation of $F$ obtained
by Theorem \ref{extension} and $\alpha_{F}$ is independent
of the choice of $F'$.

\begin{theorem} \label{Almgren map}
Two maps $F_0: X \rightarrow \Zk$ and
$F_1: X \rightarrow \Zk$
are homotopic if and only if 
 $\alpha_{F_0}=\alpha_{F_1} \in H_{p+k}(X \times M, X \times \partial M; G)$.
\end{theorem}

\begin{proof}
One direction follows immediately from
the construction: if $F_0$ and $F_1$ are homotopic, then
$\alpha_{F_0}=\alpha_{F_1}$.
Indeed, let $F:[0,1] \times X \rightarrow \Zk $
be the homotopy. By Theorem \ref{extension} 
there exists a $\delta$-localized approximation 
$F'$ of $F$ and the restrictions $F'|_{i \times X}$
are $\delta$-localized
approximations of $F_i$ for $i =0,1$.
Applying our construction to $F'$
we obtain a $(p+1+k)$-chain $B$ with $\partial B = 
A_{F'|_{1 \times X}} - A_{F'|_{0 \times X}}$.

Now we will prove the other direction.
Let $F_0':  X(q)_0 \rightarrow \Zk$ and
$F_1': X(q)_0 \rightarrow \Zk$ be two $\eps$-fine
$\delta$-localized maps that $\eps$-approximate 
$F_0$ and $F_1$ correspondingly.

Define $A_{F_1'}$ and $A_{F_2'}$ as described above.
Assume that they represent the same homology class.
We would like to construct a homotopy from $F_1'$
to $F_2'$.
It is enough to construct a sequence of maps
$F^i: X(q')_0 \rightarrow \Zk $, $q'\geq q$, which are 

1) $\eps$-fine;

2) satisfy $\F(F^i(x), F^{i+1}(x)) \leq \eps$
for all $x \in X(q')_0$;

3) $\F(F^0(x), F_1'(y) - F_0'(y)) \leq c \eps $
for $x,y $ in the same cell of $X(q)$;

4) $F^N(x)$ is the empty cycle for all $x$.

By Theorem \ref{approximation} such a sequence of maps
would guarantee existence of a continuous homotopy contracting
$F_1-F_0$.

We will define sequence $F_i$ by defining a sequence
of $(p+k)$-cycle $A_i$ in $X \times M$ and intersecting them with fibers of $proj_X$.

We have that cycles $A_{F_1'}$ and $A_{F_2'}$  are constructed in such a way that for every $\eps>0$ picking a sufficiently fine set of
discrete points $X(q'')_0$ (i.e. $q''$ sufficiently large) we can guarantee that family $\{proj^{-1}(x) \cap A_{F_j'} \}$
is $\eps$-fine.
We observe that every $(p+k)$-cycle in general position 
will have this property. This follows from PL version of transversality,
which we explain in more detail below.

Fix a fine PL structure on of $M$ 
that is $(1+\eps)$-bilipschitz to the original metric.
Fix a cell $ E_X^p \times E^n_M $, $E_X \subset X$, $E_M \subset M$.
Let $C$ be a linear $m$-cell contained in $E_X^p \times E^n_M$, $m>0$.
Let $\theta(C)$ denote the angle that $C$ makes with fibers 
of $proj_X: E_X \times E_M \rightarrow E_X$, defined as follows.
Fix a point $a$ in the interior of  $C$.
Let $P_M$ denote the $n$-plane that
passes through $a$ and contains a fiber of $proj_X$.
Let $P_C$ denote the $m$-plane that
passes through $a$ and contains $C$.
If $m>p$ let $\beta = \{ v_1, ... , v_{m-p} \}$ be a set of linearly independent
vectors contained in $ P_C \cap  P_M$, otherwise set $\beta = \{ 0 \}$.
Define $\theta(C)$ to be the minimum over all angles between 
non-zero vectors $v_1 \in P_C$  perpendicular to $span(\beta)$
and $v_2 \in P_M$ perpendicular to $span(\beta)$.
Note that this definition is independent of the choice of $a$.


We will say that a polyhedral chain (resp. relative cycle) $A^m$ 
in $X^p \times M^n$, $m \geq p$, is in general position
if the following conditions are satisfied:

\begin{enumerate}
\item $A \cap X \times \partial M$ is an $(m-1)$-dimensional 
polyhedral chain
(resp. cycle);

\item Each cell $C^i$ of $A$ is contained in the interior
of a cell $E^j$ of $X \times M$ with $j \geq i$;

\item For each cell $C^i$, $i>0$, of $A$ we have $\theta(C)>0$.
%
\end{enumerate}

\begin{lemma} \label{generic}
If $A^m$ is a polyhedral relative cycle in general position, then
there exists an open and dense subset $G \subset X$,
such that the following holds:

(a) for every $x \in G$ we have that 
$proj_X^{-1}(x) \cap A$ is a polyhedral $(m-p)$-dimensional relative cycle
and $proj_X^{-1}(x) \cap C \cap X \times \partial M$ 
is a polyhedral $(m-p-1)$-dimensional cycle.

(b) for every $\eps>0$ there exists 
$\delta>0$, so that if $x,y \in G$ and $dist(x,y) < \delta$
then $\F (proj_M(proj_X^{-1}(x)), proj_M(proj_X^{-1}(y)) )< \eps$.
\end{lemma}

\begin{proof}
Let $(A)_0$ denote the 
$0$-skeleton of $A$ and let $G \subset X$
be the union of interiors of all $n$-dimensional
faces of $X$ minus $proj_X((A)_{0})$.
By \cite[Theorem 1.3.1]{Wi66} and our angle
condition $proj_X^{-1}(x)$ is
a $(m-p)$-dimensional  polyhedral cycle for each point $x \in G$,
and given $y_i \in G$ and a piecewise linear segment $L_i$ connecting $y_i$ to $x$,
$proj_X^{-1}(L_i)$ is a polyhedral $(m-p+1)$-chain with $\Vol(L_i) \rightarrow 0$ as $y_i \rightarrow x$. 
\end{proof}




Let $B$ denote a polyhedral $(p+k+1)$-chain with 
$\partial B = A_{F_1} - A_{F_2}$.
There exists a fine subdivision of 
 $B$, so that for any $(p+k)$-cell $C$ of $B$
 and almost every $x \in X$ we have 
 that $z_C(x) = proj_X^{-1} (x) \cap C$ is a polyhedral $k$-chain
 of volume less than $\eps$.

Let $N$ denote the number of $(p+k+1)$-cells in $B$.
We can pick a large enough $q' \geq q$, so that
for every two adjacent vertices $x_1,x_2 \in X(q')_0$ and
every $(p+k)$-cell $C$ of $B$ we have

$$\F(proj_M(proj_X^{-1} (x_1) \cap C),proj_M(proj_X^{-1} (x_2) \cap C)) < \eps /N$$


Let $\{ C_i\}_{i=1}^N$ be $(p+k+1)$-cells of $B$.
Define a sequence of maps $F_i: X(q')_0 \rightarrow \Zk$
as follows:
$$F_i(x) = proj_M (proj_X^{-1}(x) \cap \partial (B \setminus \cup_{j=1}^i C_j) )$$

By construction this defines a sequence of $\eps$-close
$\eps$-fine discrete maps and by Theorem \ref{approximation} we can construct
a homotopy from $F: X \rightarrow \Zk$ given by 
$F(x) = F_1(x) - F_0(x)$ to the $0$-map.
\end{proof}

Let $F_A: \pi_m(\mathcal{Z}_k(M^n;G))
\rightarrow H_{m+k}(M^n \times S^m;G)$
be a map defined as follows. 
Given a map $f: S^m \rightarrow \mathcal{Z}_k(M^n;G)$
representing class $[f] \in \pi_m(\mathcal{Z}_k(M^n;G)) $ we set
$F_A([f]) = \alpha_f \in H_{m+k}(M^n \times S^m;G)$.

\begin{theorem} \label{groupoid}
$F_A$ is a bijective
groupoid homomorpism.
\end{theorem}

\begin{proof}
By Theorem \ref{Almgren map} if $f: S^m \rightarrow \mathcal{Z}_k(M;G)$ and $g: S^m \rightarrow \mathcal{Z}_k(M;G)$
are two homotopic maps, then cycles 
$A_f$ and $A_g$ are homologous, so the map is well defined.
It also follows from the construction of $A_f$ that $F_A$ is a homomorphsim
and injectivity of $F_A$ follows by Theorem \ref{Almgren map}. 

Every homology class in $H_{m+k}(M^n \times S^m;G)$
can be represented by a polyhedral chain $A$ in general position, 
so that for a sufficiently large $q$ discrete family 
$\{proj_M(proj_X^{-1}(x) \cap A) \}_{x \in X(q)_0}$ represents a unique homotopy class in $\pi_m(\mathcal{Z}_k(M^n;G))$.
Hence, $F_A$ is also surjective.
\end{proof}

By the Kunneth formula we have 
$H_{m+k}(M^n \times S^m;G) \cong H_{m+k}(M^n;G) \oplus H_k(M^n;G) $.
If we restrict the domain of $F_A$ to
maps $f: S^m \rightarrow \mathcal{Z}_k(M^n;G)$ with a fixed a base point cycle $z \in \mathcal{Z}_k(M;G)$ 
we obtain the following isomorphism.

\begin{corollary}[Almgren isomorphism theorem] \label{AIT}
The homotopy group $\pi_m(\mathcal{Z}_k(M^n;G); \{ [z] \})$
is isomorphic to the homology group
$H_{m+k}(M^n;G)$.
\end{corollary}

For future applications we will also need the following
Proposition, which allows us to replace a contractible 
family of relative cycles with a contractible family of 
absolute cycles.

\begin{proposition} \label{relative to absolute}
Let $F: X(q) \rightarrow \Zk$ be a contractible map.
For every $\eps>0$ there exists a map $G: X \rightarrow \mathcal{Z}_k(M, G)$,
such that $$\mathcal{F}( F(x) \llcorner int(M), G(x) \llcorner int(M) ) < \eps$$
\end{proposition}

\begin{proof}
Let $A_{F}$ denote a $(p+k)$-dimensional relative cycle in $X \times M$
corresponding to $F$.
Since $F$ is contractible, by Theorem \ref{Almgren map}
there exists a cycle $B$ with $\partial B - A_F \subset X \times \partial M$.
For $\partial B$ in general position and $q$ sufficiently large 
we can define an $\eps/2$-fine map
$G: X(q)_0 \rightarrow \mathcal{Z}_k(M, G)$ defined by
$G(x) = proj_M(proj_X^{-1}(x)) \cap \partial B$.

By Theorem \ref{approximation} we can complete $G$ to a continuous family of
absolute cycles approximating $F$.
\end{proof}

We finish this section with a counterexample to parametric isoperimetric inequality
for contractible families of cycles with integer coefficients.

\begin{proposition}
Let $N>0$. There exists a contractible family of $0$-cycles
$F_N: S^1 \rightarrow \mathcal{Z}_0(S^1;\Z)$, such that 
\begin{itemize}
    \item $\M(F_N(x)) \leq 2$ for all $x \in S^1$;
\item any family of fillings
$H:S^1 \rightarrow I_1(S^1;\Z)$ with $\partial H(x) = F_N(x)$ satisfies
$\M(H(x_0))> N$ for some $x_0 \in S^1$.
\end{itemize}

\end{proposition}

\begin{proof}
Consider $[0,1] \times S^1$ and let $\{ \gamma_i = \{t_i\} \times S^1\}$ be a collection
of $2N$ disjoint closed curves in $[0,1] \times S^1$ all oriented the same way 
and with $t_i \in (0,\frac{1}{2})$; let $\{ \beta_i = \{s_i\} \times S^1\}$ be a collection of $2N$ disjoint closed curves all oriented in the opposite way to 
$\gamma_i$ and with $s_i \in (\frac{1}{2},1)$. Observe that we can make 
a PL perturbation of curves $\gamma_i$ and $\beta_i$, 
so that fibers of the projection map 
$p:[0,1] \times S^1 \rightarrow [0,1]$ intersect $\bigcup \gamma_i \cup \beta_i$ in at most $2$ points (and when the intersection has exactly two points they have  opposite orientation).

Identifying the endpoints of $[0,1]$ we can think of curves $\gamma_i$ and
$\beta_i$ as lying in $S^1 \times S^1$. Let $X= S^1$, then for sufficiently
large $q$ we have that $\{ F(x) = p^{-1}(x) \cap (\bigcup \gamma_i \cup \beta_i): x \in X(q)_0\}$
is an $\eps$-fine discrete family of cycles.
Since $Z=\bigcup \gamma_i \cup \beta_i$ is null-homologous
as a 1-cycle in $S^1 \times S^1$ the family $F(x)$
is contractible. Every family of fillings $H$
of $F$ then corresponds to a filling $\tau$ of $Z$.
It is easy to see, using a winding number argument,
that either $p^{-1}(0)\cap \tau$
or $p^{-1}(\frac{1}{2})\cap \tau$ is a circle of 
multiplicity $\geq N$.
It follows that $\M(H(x_0)) \geq N$ for some $x_0 \in S^1$.
\end{proof}

\section{Parametric coarea inequality}\label{sec:coarea}
Let $\Omega\subset \R^3$ be a domain with $\partial \Om$ piecewise-smooth boundary
with $\theta$-corners.
Let $\Om_{\eps} = \Om \setminus E(\partial \Om \times [0,\eps]) $,
where $E$ is the ``exponential map'' from Section \ref{lem: boundary with corners},
and $\Si = \partial \Om_{\eps}$.
In this section we will prove a somewhat stronger version of Conjecture \ref{conj: coarea}
for $n=3$ and $k=1$.

\begin{theorem} \label{coarea}
Fix $\eta>0$.  For all $\eps \in (0, \eps_0)$, $\delta_1>0$ and $p\geq p_0(\Om, \eps)$ the following holds.
Let $F: X^p \rightarrow \Zone(\Omega, \partial \Omega; \Z_2)$ be
a continuous map with no concentration of mass.
There exists a map $F': X \rightarrow I_1(cl(\Om_\eps); \Z_2)$, such that 

\begin{enumerate}
    \item $\partial F'$ is a continuous family of $0$-cycles in $\partial \Om_{\eps}$ that is $\delta_1$-localized on $X(q')$ for sufficiently large $q'$;
    \item $\mathcal{F}(F(x)\llcorner  \Om_\eps, F'(x)\llcorner  \Om_\eps) < \eta$;
\item $\M(F'(x)) \leq \M(F(x) + c(n) \M(\partial \Om) (p^{\frac{1}{2}}+\frac{\M(F(x))}{\eps \sqrt{p}})$;
\item  $\M(\partial F'(x)) \leq c(n) \M(\partial \Omega) p$
\end{enumerate}

Moreover, if $F$ is a $p$-sweepout of $\Omega$, then 
$\partial F': X \rightarrow \mathcal{Z}_0(\partial \Om_{\eps}; \Z_2)$
is a $p$-sweepout of $\partial \Om_{\eps}$.
\end{theorem}

\begin{proof}

By Proposition \ref{approximation no concentration}
we can replace family $F$ with a $\delta$-localized family,
which is arbitrarily close to $F$ in flat norm, while
increasing the mass by an arbitrarily small amount. 
Hence, without any loss of generality we may assume that 
$F$ is $\delta$-localized for some very small 
(compared to $\eta$ and $\eps$) $\delta>0$.

\subsection{Proof strategy} We give a brief informal overview of the proof strategy.
For each value of $x \in X(q)_0$ we can use coarea inequality to find $s(x)$,
so that the restriction $F(x) \llcorner \Om_{\eps-s(x)}$ satisfies a good bound
on boundary mass. We then project the part sticking outside
of $\Om_{\eps}$ onto $\Sigma = \partial \Om_\eps$ and denote the resulting
chain (that will be a relative cycle in $\Om_{\eps}$) by $C_{s(x)}(F(x))$ (see precise 
definitions in subsection \ref{notation}).
This gives us a way to define the family on the $0$-skeleton $X(q)_0$.
The challenge is to extend this family to the higher dimensional skeleta of $X(q)$.

For two adjacent vertices $x$, $y$ we have that $I=C_{s(x)}(F(x))- C_{s(y)}(F(y))$
is a chain supported in $\Sigma$. Recall that $F$ is $\delta$-localized and so we 
can assume that for any two adjacent vertices the values of $s(x)$ and $s(y)$ are chosen so that 
$F(x)-F(y)$ is supported  away from $\partial \Om_{\eps-s(x)}$
and $\partial \Om_{\eps-s(y)}$. It follows that we can deform $I$ into a chain
$I' = C_{s(x)}(F(x))- C_{s(y)}(F(x))$ and that we have a good bound 
on the mass of $\partial I'$. To interpolate between $C_{s(x)}(F(x))$
and $C_{s(y)}(F(y))$ we would like to contract $I$, while controlling its mass and boundary mass. 
Note that the mass of $I$ may be as large as that of $F(x)$ or $F(y)$. So to do this we subdivide
$\Sigma$ into small triangles $\{Q_i\}$ with the goal of contracting $I$ radially triangle by triangle. 
However, we have no control over how large the mass of the intersection of $I$ with $\partial Q_i$ is
(but can assume that $F$ was slightly perturbed so that the intersection has finite mass),
so contracting $(I)_i = I \llcorner Q_i$ may increase the boundary mass by an arbitrary amount.
To solve this problem we add to each $C_{s(x)}(F(x))$ a chain $\sum_i f_i(x)$, where each
$f_i(x)$ is supported in $\partial Q_i$ and is a kind of ``filling'' of
the $0$-chain $C_{s(x)}(F(x)) \llcorner \partial Q_i$, reducing its mass
to $\leq 1$. 

\begin{figure} 
	\centering
	\includegraphics[scale=0.7]{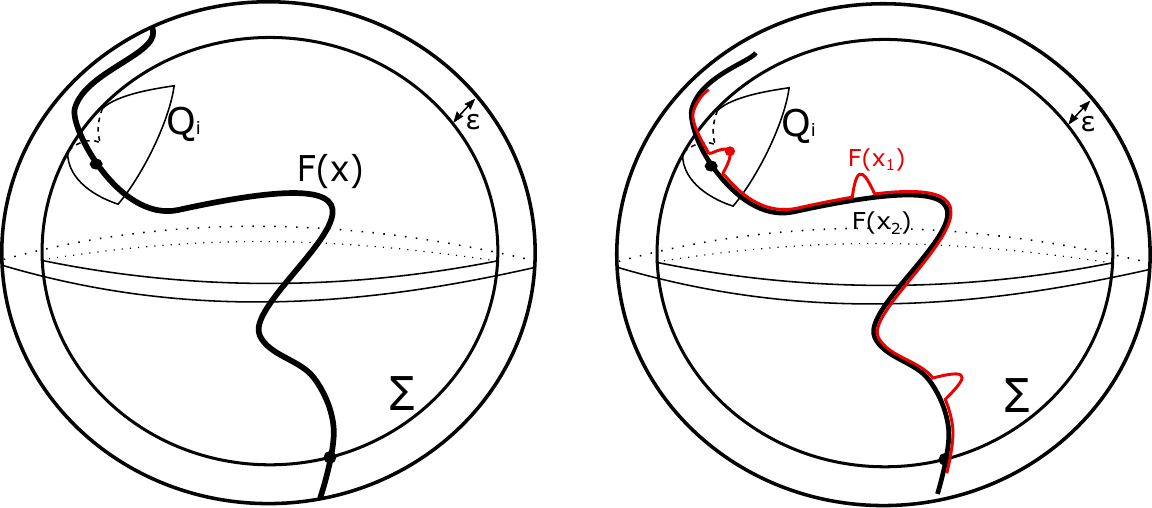}
	\caption{$F(x_1)$ and $F(x_2)$ are two cycles with $x_1, x_2 \in C \cap X(q)_0$.
	We use coarea inequality to cut them in $\Om \setminus \Om_{\eps}$.
	Then we project the part sticking out outside of $\Om_\eps$ 
	onto $\Si = \partial \Om_{\eps}$}.
	\label{fig:coarea1}
\end{figure}

With this in mind we will define $F'(x)= \sum_i C_{s(x)}(F(x))\llcorner Q_i + f_i(x) + Cone(\partial F(x) \llcorner Q_i)$ on the $0$-skeleton $X(q)_0$. The last term is the cone over the boundary of 
$C_{s(x)}(F(x))$ in $Q_i$ with vertex $q_i \in Q_i$. Adding the cone makes some of the interpolation
formulas simpler as we extend $F'$ to higher dimensional faces of $X(q)$. The interpolation process is
illustrated on Fig \ref{fig:interpolation}.

Let's ignore, for simplicity, the difference between values of $F$ on adjacent
vertices, since they are contained in a collection of tiny balls away from where we ``cut'' $F(x)$. Interpolation procedure then can be thought of as radial scaling of terms like $[C_{s(x_l)}(F(x)) - C_{s(x_j)}(F(x))] \cap Q_i +f_i(x_l)- f_i(x_j)$
for different vertices $x_l$ and $x_j$ in the 0-skeleton of a cell in $X(q)$. On Fig \ref{fig:interpolation} on the left
we see $(F'(x_4))_i = (C_{s(x_4)}(F(x)))_i+f_i(x_4) +Cone$, where $s(x_4)$ corresponds to the cut that is closest to $\partial \Om$. Then we have cuts corresponding to $s(x_1)$,
$s(x_2)$ and $s(x_3)$ indicated on the same picture. The cut corresponding
to $s(x_2)$ has mass $1$ and the other two cuts have mass $3$.
The interpolation will involve scaling the difference between these cuts as on the figure on the right and adding a portion of the cone over
$q_i$ to connect the boundaries. Inductively we will assume that
on the boundary of a cell $C$ the family of chains is of this form
and then linearly change the scaling factors to homotop the family 
in $Q_i$ to a cycle $F'(x_j)\llcorner Q_i$ 
that has the least boundary mass (in the case of the figure this
is $x_j = x_2$ or $x_j = x_4$).

\begin{figure} 
	\centering
	\includegraphics[scale=0.7]{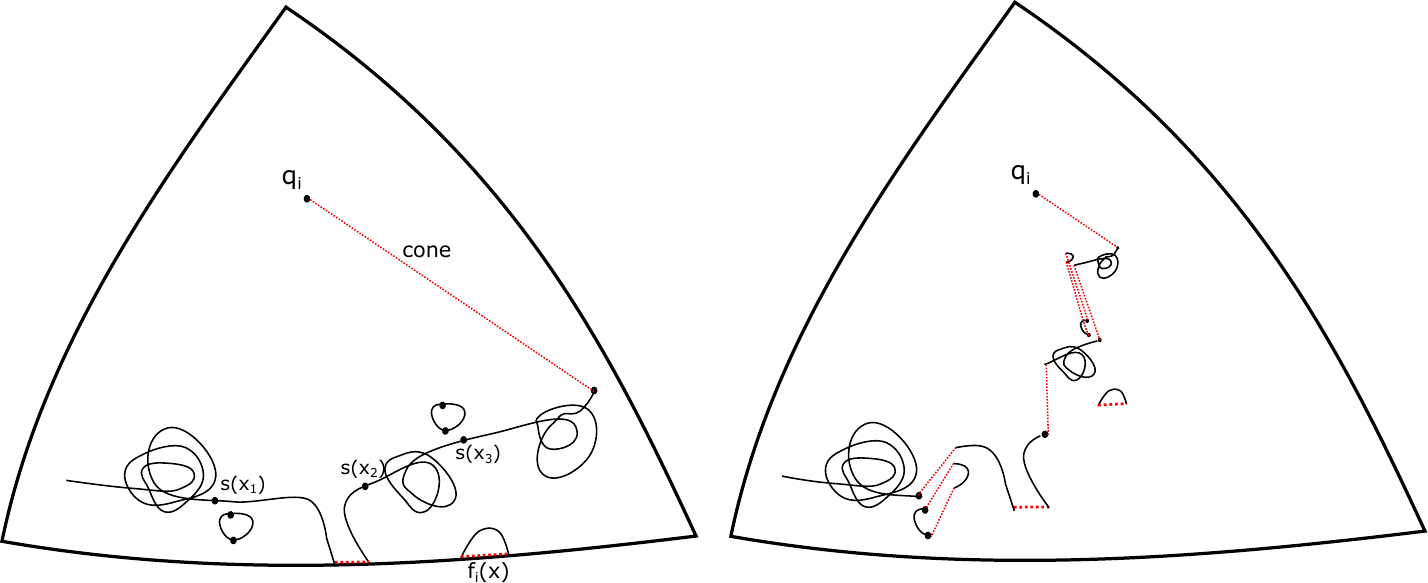}
	\caption{}
	\label{fig:interpolation}
\end{figure}

\subsection{Notation}\label{notation}
We need to make the following definitions:

\begin{itemize}
\item    Let
$\pi: \Om \setminus \Om_{\eps} \rightarrow \Si$ be the projection map.

\item
Given a $k$-chain $\tau \in I_1(\Om; \Z_2)$ and $s \in [0,\eps]$
let $L_s(\tau) = \pi(\tau \llcorner 
(\Om_{\eps -s} \setminus \Om_\eps))
\subset \Si$
and $C_s(\tau) = (\tau \llcorner \Om_\eps)+L_s(\tau) \subset cl(\Om_\eps)$.

\item
Let $[\tau]_{[s_1,s_2]}= C_{s_2}(\tau) - C_{s_1}(\tau)$.

\item For $x \in X(q)$ let $B(x)$ denote the cell of $X(q)$ of smallest
dimension, such that $x$ lies in its interior; if $x \in X(q)_0$,
then $B(x) = x$.


\item
For each vertex $x \in X(q)_0$ we apply applying coarea inequality Theorem \ref{classical coarea} (here we assume that $\eps\leq \eps(\Om)$ is sufficiently small) to find $s(x) \in [\frac{\eps}{2},\eps]$, such that
$$\M(\partial C_{s(x)}(F(x))) \leq \frac{1.5}{\eps-4\delta} \M(F(x) \llcorner cl(\Om \setminus \Om_\eps))\leq \frac{2}{\eps}  \M(F(x) \llcorner cl(\Om \setminus \Om_\eps))$$
and, moreover, we have that 
$\partial \Om_{\eps-s(x)}$
is disjoint from a $\delta$-admissible family $\{ U_i\}$.

In particular, for any $x, y \in X(q)_0$ that lie in a common
face of of $X(q)$ we have that $C_{s(x)}(F(x)-F(y))$
is an absolute cycle.


\item For $p > p_0 (\Om, \eps)$ we
can triangulate $\Sigma$ by subdividing it into $N'$ subsets $Q_i$, 
$N' \leq const \ p \M(\partial \Om)$, with disjoint interiors 
and piecewise smooth boundaries,
such that there exists a $(1+\eps^2)$-bilipschitz diffeomorphism
$P_i$ from $Q_i$
to a convex simplex $Q_i' \subset \R^2$, 
$diam (Q_i) \leq p^{-\frac{1}{2}}$ and $\M(\partial Q_i) \leq p^{-\frac{1}{2}}$.

\item 
For each  triangle $Q_i$ fix a point $q_i$ in the interior of $Q_i$.
Given 
a cycle $z$ let $Cone_i(z) = P_i^{-1}(Cone_{P_i(q_i)}(P_i(z\llcorner Q_i)))$ denote the cone over $z\llcorner Q_i$ with vertex $q_i$.

\item Let $\Psi_{\rho}^i: Q_i \rightarrow Q_i$ be the radial scaling (towards $q_i$) map
$\Psi_{\rho}^i (x) = P_i^{-1}(P_i(q_i) + \rho(P_i(x)-P_i(q_i)))$, $\rho \in [0,1]$.

\item Given a $k$-chain $\tau$, $(k-1)-$cycle $x$ and $\rho \in [0,1]$ 
define the following ``conical collar'' map:
$$\Phi_{\rho}^i(\tau,x) = Cone_i(x) - Cone_i(\Psi_{\rho}^i(x))
+ \Psi_{\rho}^i( \tau)$$
We will drop superscript $i$ whenever it is clear to which region
the map is being applied.

In the special case when $x = \partial \tau$ we have that
$\Phi_{\rho}(\tau,x)$ fixes the boundary of $\tau$, and replaces the chain with a sum  of a 
shrunk copy of itself and a piece of the cone over $\partial \tau$.
More generally, we have $\partial \Phi_{\rho}(\tau,x) = x 
- \Psi_{\rho}( x) +\partial \Psi_{\rho}( \tau) $.

\item We will use the following notation: given a chain $I$
we let $(I)_i = I \llcorner Q_i$.

\item Let $k\geq 2$. Given chains
$\tau_1, \dots, \tau_k$ and
$0 \leq \rho_1 \leq \dots \leq \rho_{k-1} \leq 1$
define a chain $\Delta_i((\tau_1, \dots, \tau_k),(\rho_1, \dots, \rho_{k-1}))$
in $Q_i$:

\begin{align*}
   \Delta_i((\tau_j)_{j=1}^k,(\rho_j)_{j=1}^{k-1})=
   & (\tau_1)_i+ \sum_{j=1}^{k-1}
   \Phi_{\rho_1 \dots \rho_j}((\tau_{j+1}-\tau_j)_i,(\partial  \tau_j)_i) \\
   &+ Cone_i(\Psi_{\rho_1...\rho_{k-1}}((\partial \tau_k)_i) )
\end{align*}
For $k=1$ we can write
$$\Delta_i(\tau) = (\tau)_i+ Cone_i((\partial \tau)_i)= \Delta_i((\tau, \tau'),(0))$$
where $\tau'$ can be any chain.

\item For each $x \in X(q)_0$ we define $f_i(x)$ to
 be a $1$-chain in $I_1(\partial Q_i; \Z_2)$ with the following property.
 Let $V_i$ denote the set of vertices of $Q_i$, then
 $\partial f_i(x)- \big( C_{t(x)}(F(x))\big)_i \llcorner \partial Q_i \subset V_i$.
%
\end{itemize}

\subsection{Definition of \texorpdfstring{$F'$}{TEXT} on \texorpdfstring{$X(q)_0$}{TEXT}}
After an arbitrarily small perturbation of the family we may assume that
 $C_{s(x)}(F(x))$ intersects $ \bigcup \partial Q_i$ transversely
 and $\partial C_{s(x)}(F(x)) \llcorner \bigcup \partial Q_i$
 is empty for each $x \in X(q)_0$.

 We define 
 \begin{equation} \label{eq:base of induction}
     F'(x) = 
 F(x) \llcorner \Omega_{\eps}+
     \sum_i \Delta_i(C_{s(x)}(F(x)))+f_i(x)
 \end{equation}

Observe that with this definition
if $x$ and $y$ are vertices in a  cell $D$ of $X(q)$, then there exists a cycle $e(x,y)$ of length $\leq \delta$ and contained in a $\delta$-admissible collection of open sets in $cl(\Om_\eps)$, s.t.
$$F'(x)= F(y) \llcorner \Omega_{\eps} + \sum_i \Delta_i(C_{s(x)}(F(y)))+f_i(x) +e(x,y)$$

\subsection{Inductive property}
Assume that we defined $F'$ on the $k$-skeleton of 
$X(q)$, so that it satisfies the following property:

\vspace{0.1in}
\noindent
(Inductive property for $k$-skeleton $X(q)_k$) For every
$p'$-dimensional cell $C^{p'}$ of $X(q)$, $k \leq p' \leq p$ and point $y \in C$ the following holds.
Let $(x_j)_{j=1}^{2^{p'}}$ denote the set of points
in $C\cap X(q)_0$ and assume that they are numbered
so that $s(x_{j_1}) \leq s(x_{j_2})$ if $j_1 < j_2$.
Then there exist functions 
\begin{itemize}
    \item $\rho_j^i:C \cap X(q)_k \rightarrow [0,1]$, $j\in \{1,...,2^{p'}-1\}$, $i\in\{1,...,N'\}$. Let 
    $$n_i(x) =\# \{\prod_{j=1}^{l} \rho_j^i(x):l\in \{1,...,2^{p'}-1\} \}-1 $$
    ($n_i(x)$ will correspond to the number of distinct collars in $Q_i$),
    then the following properties hold:
    \begin{enumerate}[i]
        \item $\sum_{i=1}^{N'} n_i(x) \leq dim(B(x)) \leq k $, moreover, for
        all $x_j \notin B(x)_0$ we have $\rho_j^i(x) = 1$ if $s(x_j)< \min_{x \in B(x)_0} \{s(x) \}$
        and $\rho_j^i(x) = 0$ if $s(x_j)= \max_{x \in B(x)_0} \{s(x) \}$;
        \item 
        \begin{gather*}
        \sum_i (1-\rho_1^i(x))\M((\partial (C_{s(x_1)}(F(y)))_i))\\
        +\sum_i\sum_j (\rho_1^i(x)...\rho_{j}^i(x)-\rho_1^i(x)...\rho_{j+1}^i(x)) \M((\partial (C_{s(x_j)}(F(y)))_i))\\
+\sum_{i} \rho_1^i(x)...\rho_{2^{p'}-1}^i(x) \M((\partial (C_{s(x_{2^{p'}})}(F(y))))_i)\leq \frac{2 \M(F(x))}{\eps}
        \end{gather*}
    \end{enumerate}

    \item $e:C \cap X(q)_k \rightarrow \Zone(cl(\Om_\eps), \Sigma;\Z_2) $ of relative cycles of length $\leq c(k) \delta$ and contained in a $c(k)\delta$-admissible collection of sets, such that

\end{itemize}
such that 
\begin{align} \label{eq:inductive step}
    F'(x) =
    & F(y)\llcorner \Omega_{\eps} + e(x)
    +\sum_{i}^{N'} \Big( \Delta_i((C_{s(x_j)}(F(y)))_{j=1}^{2^{p'}},
(\rho_j(x))_{j=1}^{2^{p'}-1}) \\
            & + f_i(x_1)+  \sum_{j=1}^{2^{p'}-1} \Psi_{\rho_1^i(x)...\rho_{j}^i(x)}^i (f_i(x_{j+1})-f_i(x_j)) \Big) \nonumber 
\end{align}
for all $x \in C \cap X(q)_k$. 
\vspace{0.1in}

We claim that our definition
(\ref{eq:base of induction}) for $k=0$ satisfies the inductive assumption.
Indeed, let $x \in X(q)_0$ and $y$ lie in a $p'$-dimensional
cell $D$ of $X(q)$ that contains $x$. 
Setting $\rho_j^i(x) =1$ for all 
$j$, such that $s(x_j) < s(x)$, and $\rho_j^i(x) =0$
for all $j$, such that $s(x_j) \geq s(x)$, we can write (\ref{eq:base of induction}) as
\begin{align*}
    F'(x) =& F(x) \llcorner \Omega_{\eps}+
\sum_{i} \Big( \Delta_i((C_{s(x_j)}(F(x)))_{j=1}^{2^{p'}},
(\rho_j(x))_{j=1}^{2^{p'}-1}) \\
            & + f_i(x_1)+ \sum_{j=1}^{2^{p'}-1} \Psi_{\rho_1^i(x)...\rho_{j}^i(x)}^i (f_i(x_{j+1})-f_i(x_j)) \Big)
\end{align*}
Let $G(x,y)$ denote the cycle obtained by replacing
$F(x)$ in the expression above with $F(y)$:
\begin{align*}
    G(x,y)=& F(y)\llcorner \Omega_{\eps}+
\sum_{i} \Big( \Delta_i((C_{s(x_j)}(F(y)))_{j=1}^{2^p},
(\rho_j(x))_{j=1}^{2^{p'}-1}) \\
            & + f_i(x_1)+ \sum_{j=1}^{2^{p'}-1} \Psi_{\rho_1^i(x)...\rho_{j}^i(x)}^i (f_i(x_{j+1})-f_i(x_j)) \Big)
\end{align*}

Then using the fact that $F(x)$ is $\delta$-localized and
our choices of $s(x)$ (that guarantee $\partial C_{s(x)}(F(x)) = 
\partial C_{s(x)}(F(y))$ if $x$ and $y$ lie in the same cell of $X(q)$)
we obtain that $e(x) = F'(x)-G(x,y)$ is a cycle of length less
than $\delta$. Property (i) of functions $\rho_j^i$ is immediate and
property (ii) 
follows since $\M(\partial C_{s(x)}(F(x)))\leq \frac{2\M(F(x))}{\eps}$.

Next we show that the inductive property implies the desired
mass and boundary mass bounds.

\begin{lemma} \label{reduction to inductive}
Suppose the inductive property is satisfied for $k=p'=p$ 
and  $y \in C^p \subset X(q)$.
Then 
\begin{enumerate}
    \item $\mathcal{F}(F(y)\llcorner  \Om_\eps, F'(y)) < \eta$;

\item $\M(F'(y)) \leq \M(F(y)) +c(n)(\frac{\M(F(y))}{\eps}\frac{1}{\sqrt{p}} +  \M(\partial \Om) \sqrt{p})$;

\item $\M(\partial F'(y)) \leq c(n) \M (\partial \Om) p $.
\end{enumerate}
\end{lemma}

\begin{proof}
The first inequality follows immediately from (\ref{eq:inductive step}).

Now we prove the second inequality. 
Observe that 
\begin{align*}
&\M\Big(\Phi_{\rho}^i\big((C_{s(x_{j+1})}(F(y))-C_{s(x_{j})}(F(y)))_i\big)\Big) \\
&\leq \rho(\M(([F(y)]_{[s(x_{j}), s(x_{j+1})]})_i) )+ (1-\rho)\M ((\partial C_{s(x_{j})}(F(y)))_i)
\end{align*}

and $$ \M (\Psi_{\rho}^i (f_i(x_{j+1})-f_i(x_j))) \leq \rho \M(\partial Q_i)$$
Unravelling the definition of $\Delta_i$ and using
property (ii) of functions $\rho_{j}^i(y)$ we obtain:
\begin{align*}
    \M(F'(y)) &\leq  \M(F(y)) + \sum_i \sum_{j=1}^{2^p-1} \Big( \M(Cone_i(\Psi^i_{\rho_1^i(y)...\rho_{j}^i(y)}((\partial C_{s(x_j)})_i)\\
    &-\M(Cone_i(\Psi^i_{\rho_1^i(y)...\rho_{j+1}^i(y)}((\partial C_{s(x_j)})_i)
               + \M(\Psi^i_{\rho_1^i(y)...\rho_{j}^i(y)} (f_i(x_{j+1})-f_i(x_j))) \Big)\\
              & \leq \M(F(y))+ \frac{2\M(F(y))}{\eps}\frac{1}{\sqrt{p}}+N'\frac{p}{\sqrt{p}} \\
              & \leq \M(F(y))+ \frac{2 \M(F(y))}{\eps}\frac{1}{\sqrt{p}} +  c(n) \M(\partial \Om) \sqrt{p}
\end{align*}

To prove inequality (3)
observe that 
\begin{gather*}
    \M(\big(\partial C_{s(x_{j})}(F(y)) +\Delta_i(C_{s(x_{j})}(F(y)),C_{s(x_{j+1})}(F(y)), \rho)+\\
    + \partial \Psi^i_\rho(f_i(x_{j+1})-f_i(x_j))\big)_i)\leq 4
\end{gather*}
where we may have one boundary point at the tip of the cone $q_i\in Q_i$ 
due to an odd number of
boundary points in $(\partial C_{s(x_{j+1})}(F(y)))_i$
and at most $3$ boundary points $[F(y)]_{[s(x_{j}), s(x_{j+1})]} \llcorner \partial Q_i +f_iF(y)$ at the vertices of triangle $Q_i$.
The total number of vertices in triangulation of
$\Sigma$ is bounded by $c'(n) p \M(\partial \Om)$.
Hence, using the first property of functions $\rho_{j}^iF(y)$, we obtain 
\begin{align*}
\M(\partial F'(y)) & \leq \sum_i^{N'} \M( (\partial F'(x))_i ) \\
    & \leq 4N' + \sum_{i,j} \M(\partial \Psi_{\rho_1^i(x)...\rho_{j}^i(x)}^i (f_i(x_{j+1})-f_i(x_j)))\\
    & \leq 4 c'(n) p \M(\partial \Om) + \sum_i n_i(y) \leq c(n) p \M(\partial \Om)
\end{align*}
\end{proof}

Now we assume that $F'$ is defined and satisfies the
inductive property on the  
$k$-skeleton $X(q)_k$. We will define
$F'$ on the $(k+1)$-skeleton and prove that it satisfies the inductive property on $X(q)_{k+1}$. 
By Lemma \ref{reduction to inductive} this will finish
the proof of the theorem.

\subsection{Inductive step}
Let $D$ be a $(k+1)$-cell in $X(q)$.
We will define map $F_D: \partial D \times [0, N']$,
such that $F_D=F'$ on $\partial D \times \{0\}$ and 
$F_D(x, N')$ is a constant map for all $x \in D$.
This immediately gives the desired extension of
$F'$ to $D$ by identifying $D$ and $\partial D \times [0, N']/\partial D \times \{N'\}$. We will show that it satisfies
the (Inductive property) and, hence, the mass and boundary
mass bounds.

Let $D_0 = \{x_1,..., x_{2^{k+1}}\}$ denote the collection
of vertices in $D \cap X(q)_0$.
Assume that the vertices are numbered 
so that $s(x_{j_1}) \leq s(x_{j_2})$ if $j_1 < j_2$.
Let $y_i(D) \in D_0$ be a vertex with
$$\M((\partial C_{s(y_i(D))}(F(x_1)))_i) = 
\min\{\M((\partial C_{s(x)}(F(x_1)))_i): x \in D_0\}$$
(Recall that $\partial C_{s(y_i)}(F(x_1)) = \partial C_{s(y_i)}(F(x))$ for all $x \in D_0$.)
By the inductive property for all $x \in \partial D$
we can write

\begin{align*}
    F'(x) =& F(x_1) \llcorner \Om_{\eps}+
\sum_{i}^{N'} \Delta_i\bigg(\Big(\big(C_{s(x_j)}(F(x_1))\big)_i \Big)_{j=1}^{2^{k+1}},
\Big(\rho_j^i(x)\Big)_{j=1}^{2^{k+1}-1}\bigg) \\
            &+e(x) + \sum_{i}^{N'} \sum_{j=1}^{2^{k+1}-1} \Psi_{\rho_1^i(x)...\rho_{j}^i(x)}^i (f_i(x_{j+1})-f_i(x_j))
\end{align*}
Now we will define functions $\rho^i_j(x,t)$, $t \in [0,N']$
and $x \in \partial D$,
as follows:
\begin{enumerate}
    \item For $t < i-1$ we set $\rho^i_j(x,t)= \rho^i_j(x)$;
    \item For $t \in [i-1,i]$
    
    \begin{equation*}
  \rho^i_j(x,t) =
    \begin{cases}
    (t-i+1)+(1-(t-i+1))\rho^i_j(x) & \text{if }s(x_j) < s(y_i(D)) \\
    (1-(t-i+1))\rho^i_j(x) & \text{if } s(x_j) \geq s(y_i(D))
    \end{cases}       
\end{equation*}
    \item For $t> i$ we set 
        \begin{equation*}
  \rho^i_j(x,t) =
    \begin{cases}
    1 & \text{if }s(x_j) < s(y_i(D))\\
    0 & \text{if }s(x_j) \geq s(y_i(D))
    \end{cases}       
\end{equation*}
\end{enumerate}

We define 

\begin{align} \label{F_D}
    F_D(x,t) =& C_{s(x_1)}(F(x_1))+
\sum_{i}^{N'} \Delta_i((C_{s(x_j)}(F(y)))_{j=1}^{2^{k+1}},
(\rho_j^i(x,t))_{j=1}^{2^{k+1}-1}) \nonumber \\
    &+e(x)+ f_i(x_1) + \sum_{i} \sum_{j=1}^{2^{k+1}-1} \Psi^i_{\rho_1^i(x,t)...\rho_{j}^i(x,t)} (x_{j+1}-x_j)
\end{align}

From $F_D$ we obtain the definitions of $\rho_{j}^i$ and $F'$ on $D$ by identifying 
$D$ and $\partial D \times [0, N']/\partial D \times \{N'\}$.
We claim that with this definition the inductive property 
will be satisfied for $F'$ on $D$.  Fix a point $y \in X(q)_0$
in a cell $C^{p'}$ of $X(q)$ that intersects $D$, $k+1 \leq p' \leq p$. 
Let $\{z_l\}$ denote the vertices of $C$ arranged, as usual,
so that $(s(z_l))$ is an increasing sequence. 
We define $\tilde{\rho}_{l}^i(x) = \rho_j^i(x)$ if $z_l = x_j$.
For the values of $l$ that correspond to vertices $z_l$ that do not lie in $D_0$ and every $x \in D \cap C $ 
we set $\tilde{\rho}_{l}^i(x)=1$ if $s(z_l) < s(x_{2^{k+1}})$
and $\tilde{\rho}_{l}^i(x)=0$ if $s(z_l) \geq s(x_{2^{k+1}})$.
For $x \in D \cap C$ consider
\begin{align*}
    G(x,y) =& F(y) \llcorner \Om_\eps +
\sum_{i} \Delta_i((C_{s(z_l)}(F(y)))_{l=1}^{2^{p'}},
(\tilde{\rho}_l(x))_{l=1}^{2^{p'}-1}) \\
            & + f_i(y_1)+ \sum_{i} \sum_{l=1}^{2^{p'}-1} \Psi_{\tilde{\rho}_1^i(x)...\tilde{\rho}_{l}^i(x)}^i (f_i(z_{l+1})-f_i(z_l))
\end{align*}
It follows from the definition (\ref{F_D}) and the inductive assumption of $e(x)$ that
$e'(x) =F'(x) - G(x,y)$ is a family of cycles of length $\leq c(k+1)\delta$ contained in a $c(k+1)\delta$-admissible collection of sets.

We would like to show that properties (i)-(ii) of functions $\tilde{\rho}^i_l$
hold. The first property follows from the inductive assumption 
and our definitions of $\rho^i_j(x,t)$ and $\tilde{\rho}^i_l(x)$.
To see that (ii) holds observe that the choice of $y_i(D)$
as the vertex with minimal $\M((\partial C_{s(y_i(D))}(F(x_1)))_i)$
implies that for $x' = (x,t)$ as defined in (\ref{F_D})
functions $\tilde{\rho}_l^i(x')$ satisfy
\begin{align*}
   &(1-\tilde{\rho}_1^i(x,t))\M((\partial (C_{s(z_1)}(F(y)))_i)) \\
    &+\sum_l (\tilde{\rho}_1^i(x,t)...\tilde{\rho}_{l}^i(x,t)-\tilde{\rho}_1^i(x,t)...\tilde{\rho}_{l+1}^i(x,t)) \M((\partial (C_{s(z_l)}(F(y)))_i))\\
& + \tilde{\rho}_1^i(x,t)...\tilde{\rho}_{2^{p'}-1}^i(x,t) \M((\partial (C_{s(z_{2^{p'}})}(F(y))))_i) \\
& \leq  (1-\tilde{\rho}_1^i(x,0))\M((\partial (C_{s(z_1)}(F(y)))_i))\\
      &  +\sum_l (\tilde{\rho}_1^i(x,0)...\tilde{\rho}_{l}^i(x,0)-\tilde{\rho}_1^i(x,0)...\tilde{\rho}_{l+1}^i(x,0)) \M((\partial (C_{s(z_l)}(F(y)))_i))\\
& + \tilde{\rho}_1^i(x,0)...\tilde{\rho}_{2^{p'}-1}^i(x,0) \M((\partial (C_{s(z_{2^{p'}})}(F(y))))_i) 
\end{align*}

Finally, we want to prove that for a given $\delta_1>0$ and some sufficiently large $q'$
map $\partial F': X(q') \rightarrow \mathcal{Z}_0(\Sigma, \mathbb{Z}_2)$ is $\delta_1-$localized on $X(q')$.
Observe that in the construction above, given any two $x,y$ in a cell $C$ of $X(q)$ and $Q_i \subset \Sigma$ we have that $(\partial F'(x) - \partial F'(y)) \llcorner Q_i$ is supported in at most $k = \sup_{v \in C_0} \M(\partial F'(x) \llcorner Q_i)$ balls of radius bounded by $c(p) dist_{\infty}(x,y) $. Hence, choosing $q'(p,\delta_1, \sup_{v \in X(q)_0} \M(\partial F'(x)) \geq q$ sufficiently large we have that map $F'$ is $\delta_1$-localized on $X(q')$.

This finishes the proof of Theorem \ref{coarea}.
\end{proof}


\section{Parametric isoperimetric inequality}\label{sec:isoperimetric}
In this section we prove a parametric
isoperimetric inequality Conjecture \ref{conj: isoperimetric} for 
$0$-cycles in a disc. One can think about this result
as a quantitative version of the Dold-Thom Theorem.

Let $Q$ denote a 2-dimensional disc of radius 1.

\begin{theorem} \label{isoperimetric}
There exists constant $c(n)>0$ with the following property.
Let $\eps>0$, $\delta>0$ and $\Phi: X^p \rightarrow \mathcal{Z}_0(Q, \partial Q; \Z_2)$ be a continuous 
contractible map.
There exists a $\delta$-localized map $\tilde{\Phi}:X^p \rightarrow \mathcal{Z}_0(Q; \Z_2)$
and $\Psi: X \rightarrow I_{1}(Q;\Z_2)$, such that

\begin{itemize}
    \item $\mathcal{F}(\Phi(x) \llcorner int(Q), \tilde{\Phi}(x) \llcorner int(Q)) < \eps$
for all $x$;
    \item $\partial \Psi(x) = \tilde{\Phi}(x)$
    for all $x$;
    \item $\M(\Psi(x)) \leq c (\M(\Phi(x))p^{-\frac{1}{2}}+p^{\frac{1}{2}})$;
    \item $\M( \tilde{\Phi}(x) ) \leq 2\M(\Phi(x))+ cp$.
\end{itemize}
\end{theorem}

We will need the following lemma:

\begin{lemma} \label{interval}
Let $F: X^p(q)_0 \rightarrow  \mathcal{Z}_0(Q; \Z_2)$
be 
an $\eps$-fine family and
let $I \subset \partial Q$ be a small interval of 
length $L$.
There exists a continuous family $F':X(q) \rightarrow  \mathcal{Z}_0(Q; \Z_2)$
with the following properties:

\begin{enumerate}
    \item $F'$ is $\eps'$-fine for $\eps' = c(p)(\eps + \sup_{x \in X(q)_0}\M(F(x))L)$;
    \item $F'(x) \llcorner int(Q) = F(x) \llcorner int(D)$ for all $x \in X(q)_0$;
\item $\M(F'(x) \llcorner I) \leq 2 p$.
\end{enumerate}

\end{lemma}

\begin{proof}
Let $z$ denote the cycle of mass 1
supported on the midpoint $e$ of $I$ and
let $B=B_L(e)$.

For every $x \in X(q)_0$ define $F'(x) = F(x) - F(x) \llcorner B + z$
if $\M(F(x) \llcorner I)$ is odd and
$F'(x) = F(x) - F(x) \llcorner B $ if $\M(F(x) \llcorner I)$ is even.

Inductively we extend $F'$ to the $k$-skeleton of 
$X(q)$.

Fix edge $E \subset X(q)_1$ with $\partial E = x-y$
and let $T$ be a $1$-chain, $\M(T) \leq \eps$, with
$\partial T = F(x) - F(y)$.
We may assume that $T$ is a finite collection of disjoint
linear arcs. We remove all arcs
of $T$ whose endpoints are contained in $I$ and replace each arc of $T$ connecting a point $a \in Q$
to $b \in I$ with an arc connecting $a$ to the midpoint $e$ of $I$. 
We call the resulting 1-chain $T'$. 
Note that by triangle inequality $T'$ satisfies

$$\M(T') \leq \M(T) + \max \{ \M(F(x)), \M(F(y)) \} L $$

We have that $\partial T' = F'(x) - F'(y)$.
We contract each arc of $T'$ one by one to obtain a family 
of 1-chains $T_t'$ and define $F'(t) = F'(x) + \partial T_t'$
for $t \in E$.
Observe that $\M(F'(t)) \leq \max \{ \M(F'(x)), \M(F'(y) \} +2$.

Now we assume that we have extended to the $k$-skeleton
with $\M(F'(x) \llcorner I) \leq 2k$
and so that the family 
is $c(k) \eps$-localized.
Let $C$ be a $(k+1)$-face.    
By the assumption that the family is localized there exists a ball $B_r(e)$, 
$r \in (L, L+ c(k) \eps)$, such that
$F'(x) \llcorner \partial B_r(e) = \emptyset$
for all $x \in \partial C$.
Let $R_t: B_r(e) \rightarrow B_{r-t}(e)$ denote a 1-Lipschitz map
defined as (in polar coordinates)
$R_t(\rho, \theta) = (\frac{(r-t-L) \rho + tL}{r-L}, \theta)$
for $\rho>L$ and $R_t(\rho, \theta)=(\rho, \theta)$. Note that $R_t$
 shrinks annulus $A(L,r,e)$ onto $A(L, r-t, e)$.
For $t \in [0, r-L]$ define
$F'(x, t) = F'(x) - F'(x) \llcorner B_r(e) + R_t(F'(x) \llcorner B_r(e)) $.
For $t = r-L$ we are guaranteed that the family $F'(x, t) \llcorner B_r(e)$
satisfies $\M(F'(x, t) \llcorner B_r(e) \llcorner \partial Q^2) \leq 2k +2$.
Then we can apply radial contraction to the centerpoint $e$.

This contracts the family $F'(x) \llcorner B_r(e)$, $x \in \partial C$,
to a single cycle. We contract $F'(x) $ outside of  $B_r(e)$ in the usual way
(see Proposition \ref{approximation_0}).
\end{proof}

\begin{lemma} \label{linear}
Suppose $F: X^p \rightarrow \mathcal{Z}_0(Q^2; \Z_2) $, 
and there exists $I \subset \partial Q^2$, such that $\M(F(t) \llcorner I) \leq K$.
There exists a continuous family of 1-chains
$\{\tau_t \}$, such that 

1. $\partial \tau_t  - F(t) \subset \partial Q$;

2. $\M(\tau_t) \leq 2( \M(F(t) \llcorner int(Q)) +K)$

3. $\M(\partial \tau_t) \leq 2 \M(F(t) \llcorner int(Q) ) + 2K$
\end{lemma}

\begin{proof}
Pick a point $a \in \R^2$ outside of disc $Q$, such that two 
tangent lines from $a$ to $\partial Q$ touch $\partial Q$
at the endpoints of interval $I$.

Given a point $q \in Q$ there is a unique line
$L_q$ passing through $a$ and $q$.
Let $H(q) = L_q \cap (\partial Q \setminus I)$.
We define
$\tau_t = Cone_a(H(F(t))) - Cone_a(F(t))$.
\end{proof}

Now we can prove Theorem \ref{isoperimetric}.
By Proposition \ref{relative to absolute} we can replace $\Phi$
with an $\eps$-close family of absolute cycles $\Phi': X \rightarrow \mathcal{Z}_0(Q;\Z_2)$
and by Lemmas \ref{linear} and \ref{interval} we may assume that there
exists a family of chains $\tau: X \rightarrow I_1(Q;\Z_2)$, such that
$\partial \tau(x) = \Phi'(x)$. Note that the family will have the desired
bound for the 
mass in the boundary.

Observe that for each $x$ the chain $\tau$
consists of finitely many interval segments each
lying on a ray from a fixed point $a$ (point $a$ lies outside of $Q$).
 We subdivide 
$Q$ into $p$ regions $\{Q_i \}$ of
boundary length and diameter $\sim \frac{1}{\sqrt{p}}$.

We perform a ``bend-and-cancel'' construction (\cite{Gu1}, \cite{Gu2}) to 
reduce the length of chains $\tau_t$
by pushing them into the 1-skeleton of the subdivision
$\cup \partial Q_i$.

For each $Q_i \subset Q$ let
$B_{2\eta}(p_i) \subset Q_i$ be a ball 
chosen so that every line passing 
through the point $a$
(from the proof of Lemma \ref{linear})
intersects at most
one ball $B_{2\eta}(p_i)$.
Let $proj_{p_i}:  Q_i \setminus p_i \rightarrow \partial Q_i$ denote the projection map from 
the point $p_i$ and define a piecewise linear map
$P_i: Q_i \rightarrow Q_i$, 
such that $P_i = proj_{p_i}$ on $Q_i \setminus B_{2\eta}(p_i)$
(in particular, $P_i(Q_i \setminus B_{2\eta}(p_i)) = \partial Q_i$)
and $P_i(y) = y$ for all $y \in B_\eta(p_i)$. Note that $P_i$ can be chosen
so that for any line $l$ passing through $Q_i$ we have
$\M(P_i(l \cap Q_i)) \leq 2 diam (Q_i) \leq
\frac{2}{\sqrt{p}}$.

Since $P_i$ is the identity map on $\partial Q_i$
we can define a map $P: Q \rightarrow Q$ by
$P(x) = P_i(x)$ for $x \in Q_i$.

Given a $0$-cycle $z \subset Q$ define a continuous map 
$l_i: \mathcal{Z}_0(Q; \Z_2) \rightarrow I_1(Q_i; \Z_2) $ by
setting
$l_i(z)$ to be the union of linear arcs
connecting each point of $z \llcorner Q_i$ to the corresponding
point of $P_i(z \llcorner Q_i)$. Let $l(z)= \sum_i l_i(z) $.

Define $\Psi(x) = P(\tau(x)) + l(\Phi'(x))$

We observe that $\partial \Psi(x) \llcorner int(Q) = \Phi'(x)$
 and 
 \begin{gather*}
     \M(\Psi(x)) \leq \M(P(\tau(x)) + \M(l(\Phi'(x))) \\
     \leq \M(\cup \partial Q_i) +\frac{2 \M(\Phi'(x))}{\sqrt{p}} + \frac{\M(\Phi'(x))}{\sqrt{p}}\\
     \leq c(\sqrt{p}+ \frac{\M(\Phi'(x))}{\sqrt{p}})
 \end{gather*}
 This finishes the proof of Theorem \ref{isoperimetric}.
 
 We now prove Conjecture \ref{conj: isoperimetric} for $k=0$ and $n=2$.
 
 \begin{theorem} \label{thm: isoperimetric0}
 Let $\Omega$ be a  2-dimensional connected simply connected manifold, $\partial \Om$
 piecewise smooth boundary with $\theta$-corners, $\theta \in (0, \pi)$, $L>0$, $p \in \mathbb{N}$.
There exist constants $c(\Omega)>0$ and $\delta(L,p,\Omega)>0$ with the following property.
Let $F: X^p \rightarrow \mathcal{Z}_{0}(\Om, \partial \Om; \Z_2)$ be a continuous contractible $\delta$-localized
$p$-dimensional family with $\sup_{x \in X} \M(F(x)) \leq L$.
Then there exists map
 $H: X \rightarrow I_{1}(\Om;\Z_2)$, such that

\begin{itemize}
    \item $\partial H(x) - F(x) \subset \partial \Om$
    for all $x$;
    \item $\M(H(x)) \leq c(\Om) \big(p^{\frac{1}{2}}+
   \M(F(x))p^{-\frac{1}{2}} \big)$.
\end{itemize}
\end{theorem}
 \begin{proof}
  Let $F: X \rightarrow \mathcal{Z}_0(\Om; \Z_2)$. Composing with a bilipschitz homeomorphsim
 $\Pi$ between $\Om$ and disc $Q$ we obtain a family $F'$ of $0$-cycles in $Q$.
 We apply Theorem \ref{isoperimetric} with $\eps_i\rightarrow 0$
 to obtain a sequence of  $\delta$-localized maps $\tilde{F}_i$
 converging uniformly to $F'$ and a corresponding family of fillings
  $\Psi_i$ of $\tilde{F}_i$.
  
  Pick $i$ large enough so that $\mathcal{F}(\tilde{F}_i(x),F'(x)) < \frac{\delta}{2} $.
  We will define a family $\tilde{H}(x)$ with $\partial \tilde{H}(x) = F'(x) - \tilde{F}_i(x)$
  and $\M (\tilde{H}(x)) \leq c_0(p) L \delta $.
  We define $\tilde{H}(x)$ inductively over the skeleton of $X$. On each vertex $x \in X_0$ 
  define $\tilde{H}(x)$ to be the area minimizing filling of $F'(x) - \tilde{F}_i(x)$. 
  By the flat distance bound there exists a $\delta$-admissible collection of open sets $\{ U_i^x\}$,
  such that $\tilde{H}(x)$ is supported in $\cup U_i$.
  
  Now we prove the inductive step. Suppose we defined $\tilde{H}$ on the $k$-skeleton of $X$ and
  for each $k$-cell $C$ the family of chains $\{\tilde{H}(x)\}_{x \in C}$ is supported in a
  $c'(k)\delta$-admissible collection of sets $\{U^C_i\}$ and $\M(H(x)) \leq c_0(k) \delta \sup_{y \in X} \M(\tilde{H}(y))$ for some constants $c_0$ and $c'$ that depend only on $k$. Let $D$ be a $(k+1)$-cell in $X$. By Lemma \ref{two admissible} there exists a constant
  $c'(k+1)>0$ and a $c'(k+1)\delta$-admissible collection of sets $\{U^D_i\}$,
  such that family of cycles $\{ F'(x) - \tilde{F}_i(x)\}_{x \in D}$ and 
  family of chains $\{\tilde{H}(x)\}_{x \in \partial D}$ are supported in $\cup U^D_i$.
  Observe that for each $i$ we can define a continuous family of conical fillings
  $\tau_i(x)$, $x \in D$, of $(F'(x) - \tilde{F}_i(x)) \llcorner U_i^D$ in $U_i^D$. (``Conical fillings'' means a collection of linear segments from the support of $(F'(x) - \tilde{F}_i(x)) \llcorner U_i^D$ to the center point of $U_i^D$.)
  It will be convenient to fix polar coordinates $(y,t)$, $y \in \partial D$, $t \in [0,1]$, on $D$, with $(y,1) = y$ and $(y,0)$ the center point of cell $D$.
  Let $h_i(y,t)$ denote the radial contraction towards the center point of $U_i^D$
  of relative cycle $\tau_i(y) - \tilde{H}(y)\llcorner U^D_i$, so that
  $h_i(y,1) = \tau_i(y) - \tilde{H}(y)\llcorner U^D_i$ and $h_i(y,0) = 0$. 
  Define $\tilde{H}(y,t) = \sum_i \tau_i(y) - h_i(y,t)$. It is straightforward to check, using the inductive assumption, that it follows from this definition that $\tilde{H}(x)$ is continuous on $D$, $\partial \tilde{H}(x) =  F'(x) - \tilde{F}_i(x)$ and
  $$\M(\tilde{H}(y,t)) \leq 2\M(\tilde{H}(y))+ c'(k+1) \M(F'(x)) \delta \leq c_0(k+1)L \delta $$
  for sufficiently large $c_0(k+1)$. This finishes the construction of $\tilde{H}$.

 Composing $H' = \Psi_i + \tilde{H}$ with $\Pi^{-1}$ we obtain the desired family $H$ of
 fillings in $\Om$. The upper bound for $\M(H(x))$ follows 
 from the upper bound for $\M(Psi_i(x)$ and choosing $\delta = \delta(p,L, \Om)$ sufficiently small.
 \end{proof}

Observe, that from Theorems \ref{thm: isoperimetric0} and \ref{coarea}
one can deduce the proof of Conjecture \ref{conj: isoperimetric} for $k=1$ and $n=3$.
We omit the details as this result is not used in the proof of the Weyl law.

\section{Proof of the Weyl law} \label{sec:weyl}
\begin{theorem} \label{weyl general}
Assume that Conjecture \ref{conj: coarea} holds for all families of $k$-cycles in domains in 
$\R^n$ and Conjecture \ref{conj: isoperimetric} holds for all families of $(k-1)$-cycles
in domains in $\R^{n-1}$. Then the Weyl law for $k$-cycles in $n$-manfolds holds:

For every compact Riemannian manifold $M$ (possibly with boundary)
$$\lim_{p \rightarrow \infty} \frac{\omega_p^k(M)}{p^{\frac{n-k}{n}}}= a(n,k) \Vol(M)^{\frac{k}{n}}$$
\end{theorem}

Combined with Theorems \ref{coarea} and \ref{isoperimetric} we obtain

\begin{corollary} \label{weyl}
Weyl law holds for 1-cycles in 3-manifolds:

For every compact Riemannian 3-manifold $M$ (possibly with boundary)
$$\lim_{p \rightarrow \infty} \frac{\omega_p^1(M)}{p^{\frac{2}{3}}}= a(3,1) \Vol(M)^{\frac{1}{n}}$$
\end{corollary}

\subsection{Proof of Theorem \ref{weyl general}}
In \cite{LMN} it was proved that for any compact contractible domain
$U \subset \R^n$
we have
$$\lim_{p \rightarrow \infty} \frac{\omega_p^k(U)}{p^{\frac{n-k}{n}}}= a(n,k)Vol(U)^\frac{k}{n}$$
where $a(n,k)$ is a constant that depends only on $n$ and $k$.

Let $M$ be a compact $n$-manifold and consider sequence
$\{\omega_p^k\}$ of $k$-dimensional $p$-widths of $M$.
It is known that for some constants $a_1(n)\leq a_2(n)$ we have
$$0<a_1 \Vol(M)^{\frac{k}{n}} =\liminf_{p \rightarrow \infty} \frac{\omega_p^k}{p^{\frac{n-k}{n}}}
\leq \limsup_{p \rightarrow \infty} \frac{\omega_p^k}{p^{\frac{n-k}{n}}}=
a_2 \Vol(M)^{\frac{k}{n}} < \infty$$
We will show that $a_2 \leq a(n,k) \leq a_1$ which implies Theorem \ref{weyl general}.

The inequality $a_1 \geq a(n,k)$ was proved in \cite[Theorem 4.1]{LMN}. 
It remains to prove $a_2 \leq a(n,k)$.

Fix $\eps>0$. 
Fix a triangulation $T$ of $M$, so that each $n$-dimensional simplex $U_i$,
$i=1,...,N$,
in the triangulation is 
$(1+\eps)$-bilipschitz homeomorphic
to a domain with $\theta$-corners in $\R^n$, for some fixed $ \theta \in (0, \frac{\pi}{2})$. 

Let $\tilde{U_i} \subset \R^n$ denote the images of $U_i$
under $(1+\eps)$-bilipschitz homeomorphism and assume that they lie at a distance
greater than $1$ from each other. 
We connect sets $\tilde{U_i}$ by tubes of very small volume to obtain
a connected set $V \subset \R^n$.
By Weyl law for domains in $\R^n$ for all sufficiently large $p$
there exists a $p$-sweepout $\Phi:X^p \rightarrow \mathcal{Z}_k(V;\Z_2)$ of $V$ by $k$-cycles of mass
bounded by $a(n,k) ((1+2 \eps)Vol(M))^\frac{k}{n}p^{\frac{n-k}{n}}$.
Let $\tilde{\Phi}_i: X \rightarrow \mathcal{Z}_k(\tilde{U}_i;\Z_2)$ denote the restriction of $\Phi$ to $\tilde{U}_i$ and let 
$\Phi_i: X \rightarrow \mathcal{Z}_k(U_i;\Z_2)$ denote the map obtained by 
composing with bilipschitz 
homeomorphism from $\tilde{U_i}$ to $U_i$.

Given $U \subset M$ with $\theta$-corners and a sufficiently small $\eta> 0$,
we can define a $(1+ c(U)\eta )$-bilipschitz 
homeomorhpism $\Delta^U_{\eta}: U_{\eta} \rightarrow U$.
(Recall that $U_\eta$ is the region obtained by removing a small neighbourhood
of $\partial U$ from $U$, $U_\eta= U\setminus E([0,\eta] \times \partial U)$,
where $E$ is as in Lemma \ref{lem: boundary with corners}).
Choose a sequence of small positive numbers $\eta_1 >...>\eta_{N-1} >0$, so that 
the maps $\Delta^{V_m}_{\eta_{m},0}: (V_m)_{\eta_i} \rightarrow V_m $
are well-defined, where $V_m = \bigcup_{j=1}^m U_j$.
Since $p$ can be chosen arbitrarily large,
we assume that $p^{-\frac{1}{n-1}}<\eta_{N-1}$.

We will define a sequence of maps
$F_m:X \rightarrow \mathcal{Z}_k(V_m;\Z_2)$, satisfying
\begin{align} \label{eq: weyl inductive step}
        \M(F_{m}(x)) & \leq  (1+c_m\eta_1)\sum_{i=1}^{m} \M(\Phi_i(x))  \\
                    & + c_{m} p^{\frac{n-k-1}{n-1}} + 
                    \frac{c_m}{\eta_{m-1}} \sum_{i=1}^{m} \M(\Phi_i(x)) p^{-\frac{1}{n-1}} \nonumber
\end{align}
(set $\eta_0 =1$).

For $m=1$ we set $F_1=\Phi_1$. Assume by induction that we defined 
$F_{m-1}$ satisfying the mass bound (\ref{eq: weyl inductive step}).
If $V_{m-1}$ and $U_m$ have disjoint boundaries,
then we set $F_m = F_{m-1} + \Phi_m$.

Otherwise, let $S = \partial V_{m-1} \cap \partial U_m$.

We apply the coarea
inequality Conjecture \ref{conj: coarea} to maps $F_{m-1}$ and $\Phi_m$ to obtain
$\delta$-localized maps $\Phi_m': X \rightarrow \mathcal{Z}_k((U_m)_{\eta_{m-1}};\Z_2)$
and $F_{m-1}': X \rightarrow \mathcal{Z}_k((V_{m-1})_{\eta_{m-1}};\Z_2)$,
such that, assuming that $p$ is sufficiently large,
\begin{align*}
\M(\Phi_m'(x)) & \leq \M(\Phi_m(x))+ c(U_m) p^{\frac{n-k-1}{n-1}}
    +c(U_m) \frac{\M(\Phi_m(x))}{\eta_{m-1}}p^{-\frac{1}{n-1}}\\
 \M(F_{m-1}'(x)) & \leq \M(F_{m-1}(x) )+ c(V_{m-1}) p^{\frac{n-k-1}{n-1}}
    +c(V_{m-1}) \frac{\M(F_{m-1}(x))}{\eta_{m-1}}p^{-\frac{1}{n-1}}\\
    & \leq \sum_{i=1}^{m-1} \M(\Phi_i(x))
         + c_{m-1}' p^{\frac{n-k-1}{n-1}} + c_{m-1}' 
        \sum_{i=1}^{m-1} \frac{\M( \Phi_i(x))}{\eta_{m-1}}p^{-\frac{1}{n-1}} \\
\M(\partial \Phi_m'(x)) & \leq c(U_m) \big(\frac{\M(\Phi_m(x))}{\eta_{m-1}}+ p^{\frac{n-k}{n-1}}\big)\\
\M(\partial F_{m-1}'(x)) & \leq c_{m-1}' \big(\sum_{i=1}^{m-1} \frac{\M(\Phi_i(x))}{\eta_{m-1}} +
p^{\frac{n-k}{n-1}}\big)
\end{align*}
Here $c_{m-1}'> c_{m-1}$ is a suitably chosen constant. Note that we used 
$p^{-\frac{1}{n-1}}<\eta_{m-1}$ to simplify the bounds for
mass and boundary mass of $F_{m-1}'(x)$.

Moreover, by Conjecture \ref{conj: coarea} we have that families $\{\partial F_{m-1}'(x) \}$ and $\{\Phi_m'(x) \}$ are $p$-sweepouts of boundaries of the corresponding
regions. Define $\overline{F}_{m-1} = \Delta^{V_{m-1}}_{\eta_{m-1}}(F_{m-1}'(x))$
and $\overline{\Phi}_m(x) = \Delta^{U_{m}}_{\eta_{m-1}}(\Phi_{m}'(x))$.

Observe that since both $\{\overline{F}_{m-1}(x) \llcorner S\}$ and
$\{ \overline{\Phi}_m(x) \llcorner S \}$ are $p$-sweepouts of  $S$
by relative $(k-1)$-cycles,
 the family  $$\{ f_{S}(x) = \Big(\partial \overline{\Phi}_m(x)+ \partial \overline{F}_{m-1}(x) \Big)\llcorner S  \}$$
is a contractible family of relative $(k-1)$-cycles in $S$.

By the isoperimetric inequality Conjecture \ref{conj: isoperimetric}
applied to $f_S$
there exists a family $\Psi_S: X \rightarrow I_k(S;\Z_2)$ with 
$\partial \Psi_S(x) - f_S(x)$ supported in $\partial S$
and satisfying
\begin{gather*}
    \M(\Psi_S(x))\leq c(S)p^{\frac{n-k-1}{n-1}}
    +\big(\M( \partial \overline{F}_{m-1}(x))+ \M(\partial \overline{\Phi}_m(x))\big)p^{-\frac{1}{n-1}}\\
    \leq c_m' \big( p^{\frac{n-k-1}{n-1}}+
    \sum_{i=1}^m \frac{\M(\partial \Phi_i(x))}{\eta_{m-1}} p^{-\frac{1}{n-1}}\big) 
\end{gather*}
Here $c_m'$ is defined in terms of
$c_{m-1}'$ and $c(S)$.

We define 
$$F_{m}(x) = \overline{F}_{m-1}(x) +\overline{\Phi}_m(x)+\Psi_S(x)  $$
Observe that $F_m$ is a continuous family of relative cycles
and a $p$-sweepout of $V_m$. (That fact that this is
a $p$-sweepout follows by restricting family to a small ball
in $U_m$ and observing that in that small ball it is homotopic
to the restriction of $\tilde{\Phi}_m$).
It is straightforward to check that the inductive assumption for the mass bounds
will be satisfied for some sufficiently large
$c_m \geq c_m'$.

For $m=N$ we obtain a $p$-sweepout $F_N$ of $M$ with
\begin{gather*}
    \frac{\M(F_N(x))}{p^{\frac{n-k}{n}}} 
    \leq (1 +c(M) \eta_1^k) \big(\frac{\M(\Phi(x))}{p^{\frac{n-k}{n}}}
    +c(M) \frac{ \M(\Phi(x)) }{ \eta_{N-1} p^{\frac{n-k}{n}+\frac{1}{n-1}} }\big) \\
    \leq a(n,k) (1 +c(M) \eps^{k})(1+ c(M) \eta_{N-1}^k) Vol(M)^\frac{k}{n} +
    c(\eps,\eta_{N-1},M)\big(p^{-\frac{1}{n-1}} \big)
\end{gather*}
Observe that as $p \rightarrow \infty$ the second term goes to $0$. 
Since $\eta_1$ and $\eps$ can be chosen to be arbitrarily small we conclude that $a_2 \leq a(n,k)$. This finishes the proof
of Theorem \ref{weyl general}.

%
%

\bibliography{bib}
\bibliographystyle{abbrv}

\end{document}